\documentclass{article}
\usepackage[utf8]{inputenc}
\usepackage{amsmath}
\usepackage{amssymb}
\usepackage{amsthm}
\usepackage{appendix}
\usepackage{enumerate}
\usepackage{mathrsfs}
\usepackage{scrextend}
\usepackage{mathtools}
\usepackage[hyperfootnotes=false]{hyperref}
\usepackage{cleveref}
\usepackage[multiple]{footmisc}
\usepackage[a4paper, total={6in, 9in}]{geometry}
\usepackage{scalerel,stackengine}
\usepackage{tikz}
\usepackage{float}
\stackMath
\newcommand\reallywidehat[1]{%
\savestack{\tmpbox}{\stretchto{%
  \scaleto{%
    \scalerel*[\widthof{\ensuremath{#1}}]{\kern-.6pt\bigwedge\kern-.6pt}%
    {\rule[-\textheight/2]{1ex}{\textheight}}
  }{\textheight}%
}{0.5ex}}%
\stackon[1pt]{#1}{\tmpbox}%
}
\DeclareMathOperator*{\E}{\mathbb{E}}
\allowdisplaybreaks

\theoremstyle{plain}
\newtheorem{theorem}{Theorem}
\newtheorem{lemma}[theorem]{Lemma}
\newtheorem{proposition}[theorem]{Proposition}
\newtheorem{corollary}[theorem]{Corollary}

\theoremstyle{definition}
\newtheorem{definition}[theorem]{Definition}

\newtheorem{remark}[theorem]{Remark}

\numberwithin{theorem}{section}

\newcommand{\cL}{\mathcal{L}}

\title{The Erdős--Moser sum-free set problem via improved bounds for $k$-configurations}
\author{Adrian Beker\footnote{Department of Mathematics, Stanford University, Stanford, CA 94305. Email: \nolinkurl{abeker@stanford.edu}}}
\date{\today}

\begin{document}

\maketitle

\begin{abstract}
A \emph{$k$-configuration} is a collection of $k$ distinct integers $x_1,\ldots,x_k$ together with their pairwise arithmetic means $\frac{x_i+x_j}{2}$ for $1 \leq i < j \leq k$. Building on recent work of Filmus, Hatami, Hosseini and Kelman on binary systems of linear forms and of Kelley and Meka on Roth's theorem on arithmetic progressions, we show that, for $N \geq \exp((k\log(2/\alpha))^{O(1)})$, any subset $A \subseteq [N]$ of density at least $\alpha$ contains a $k$-configuration. This improves on the previously best known bound $N \geq \exp((2/\alpha)^{O(k^2)})$, due to Shao. As a consequence, it follows that any finite non-empty set $A \subseteq \mathbb{Z}$ contains a subset $B \subseteq A$ of size at least $(\log|A|)^{1+\Omega(1)}$ such that $b_1+b_2 \not\in A$ for any distinct $b_1,b_2 \in B$. This provides a new proof of a lower bound for the Erdős--Moser sum-free set problem of the same shape as the best known bound, established by Sanders.
\end{abstract}

\section{Introduction}\label{sec:intro}

The motivation for this paper comes from the following old question of Erdős and Moser \cite[p.\ 187]{erdos}, which was recently reiterated by Green in his list of open problems \cite[Problem 2]{green-open-problems} (see also \cite{bloom-erdosproblems-787}). Given a finite non-empty set of integers $A$, what is the largest size of a subset $B \subseteq A$ which is sum-free with respect to $A$? Here, we say that $B$ is \emph{sum-free with respect to $A$} to mean that $b_1+b_2 \not\in A$ for any distinct $b_1,b_2 \in B$. Letting $M(A)$ denote the answer to this question, we are interested in the behaviour of $\phi(n)$, the minimum of $M(A)$ as $A$ ranges over all sets of integers of size $n$. Note that the above requirement that $b_1,b_2$ be distinct is necessary in order for the problem to be non-trivial. Indeed, the example $A = \{2^j \mid j \in [n]\}$ shows that dropping the distinctness assumption would result in $\phi$ being constantly equal to $1$.

We briefly review the history of progress on this problem. A simple example showing that $\phi(n) \leq \frac{1}{3}n+O(1)$ was given by Erdős \cite{erdos}. This was subsequently improved by Selfridge \cite[p.\ 187]{erdos} to $\phi(n) \leq (\frac{1}{4}+o(1))n$. A sublinear upper bound on $\phi(n)$ was first established by Choi \cite[p.\ 190]{erdos}, who later improved this \cite{choi} to a power-saving bound\footnote{In fact, Choi also observed that the present formulation of the problem is equivalent to the one initially proposed by Erdős and Moser, in which the reals were taken as the ambient group.} of the form $\phi(n) \leq n^{2/5+o(1)}$. The error term in the exponent was refined by Baltz, Schoen and Srivastav \cite{baltz-schoen-srivastav}, who proved that $\phi(n) = O((n\log n)^{2/5})$. Ruzsa \cite{ruzsa-sum-free} was the first to prove that $\phi(n)$ grows subpolynomially in $n$, namely $\phi(n)\leq \exp(O(\sqrt{\log n}))$. This is the best known upper bound to date.

In the other direction, progress has been somewhat slower. Erdős and Moser showed that $\phi(n) \to \infty$ as $n \to \infty$ and Klarner obtained the quantification $\phi(n) = \Omega(\log n)$ (see \cite{erdos}). However, the first published proof of a non-trivial lower bound seems to be that of Choi \cite{choi}, who showed using a greedy argument that $\phi(n) \geq \log_2n$ (see also \cite[Theorem 6.2]{tao-vu} for a proof of the slightly weaker bound $\phi(n) \geq \log n+O(1)$). Using a more nuanced greedy strategy, Ruzsa \cite{ruzsa-sum-free} was able to obtain the constant factor improvement $\phi(n) > 2\log_3n-1$. Soon afterwards, a significant advance was made by Sudakov, Szemerédi and Vu \cite{sudakov-szemeredi-vu}, who proved the first superlogarithmic lower bound, namely $\phi(n) \geq g(n)\log n$ with $g(n) = (\log^{(5)}n)^{1-o(1)}$ (here, $\log^{(k)}$ denotes the $k$-fold logarithm). This was later improved by Dousse \cite{dousse} to $g(n) = (\log^{(3)}n)^{\Omega(1)}$ and Shao \cite{shao} to $g(n) \geq (\log^{(2)}n)^{1/2-o(1)}$. Finally, in a recent tour de force of additive combinatorics, Sanders \cite{sanders} showed that $g(n) = (\log n)^{\Omega(1)}$ is permissible.

We take a moment to further discuss the series of works starting from the Sudakov--Szemerédi--Vu breakthrough since these are particularly relevant to us. By making extensive use of the Balog--Szemerédi--Gowers--Freiman machinery, Sudakov, Szemerédi and Vu ultimately reduce the problem to that of finding $k$-configurations in dense sets of integers. Herein, if $G$ is an abelian group such that the map $G \to G$, $x \mapsto 2x$ is an automorphism\footnote{In our case, either the additive group of the rationals or a finite abelian group of odd order.} and $k \geq 2$ is an integer, a \emph{$k$-configuration} is defined to be a collection $\mathcal{C}$ of $k$ elements $x_1, \ldots, x_k \in G$ together with their pairwise arithmetic means $\frac{x_i+x_j}{2}$ for $1 \leq i < j \leq k$. We say that $\mathcal{C}$ is \emph{generated by} $x_1,\ldots,x_k$; $\mathcal{C}$ is said to be \emph{non-degenerate} if $x_1, \ldots, x_k$ are pairwise distinct.

Given $\alpha \in (0,\frac{1}{2}]$ and $k \geq 2$, let $F(\alpha,k)$ denote the smallest positive integer $N$ such that for any $n \geq N$, any subset of $[n]$ of density at least $\alpha$ contains a non-degenerate $k$-configuration. Sudakov, Szemerédi and Vu observed that an arithmetic progression of length $2k-1$ contains the $k$-configuration generated by the elements at odd indices, so appeal to Gowers' bound \cite{gowers-szemeredi} for Szemerédi's theorem yields
\[F(\alpha,k) \leq \exp(\exp(\alpha^{-\exp(\exp(O(k)))})).\]
However, as was first observed by Dousse \cite{dousse}, this is overkill in that one doesn't need the full strength of Szemerédi's theorem in order to obtain bounds for $F$. Indeed, the system of linear forms defining a $k$-configuration has complexity $1$ in the sense of Green and Tao \cite[Definition 1.5]{green-tao}. Thus, to get a handle on $k$-configurations, one can use the $U^2$-norm (i.e.\ Fourier analysis) rather than the higher-order $U^{k-1}$-norm. Implementing this idea analogously to Roth's proof of his theorem on three-term progressions \cite{roth}, Dousse obtains the bound
\[F(\alpha,k) \leq \exp(\exp(\alpha^{-O(k^2)})).\]
Shao's improvement arises, roughly speaking, from the use of Bohr sets in analogy to Bourgain's proof of Roth's theorem \cite{bourgain}. He proves in \cite{shao} that
\[F(\alpha, k) \leq \exp(\alpha^{-O(k^2)}).\]
In contrast to previous works, Sanders' bounds on $\phi$ aren't obtained by improving the bounds on $F$. Instead, Sanders bypasses the study of $k$-configurations by exploiting certain weaknesses in the initial reduction to this problem. Hence, even though we now have a better understanding of the Erdős--Moser sum-free set problem, there remains the problem of improving bounds for sets lacking $k$-configurations.

In the meanwhile, there have been major advances in our understanding of sets lacking three-term progressions. In a remarkable breakthrough, Kelley and Meka \cite{kelley-meka} established that any subset of $[N]$ with this property has density at most\footnote{This has subsequently been improved to $\exp(-\Omega((\log N)^{1/9}))$ by Bloom and Sisask \cite{bloom-sisask-improvement}.} $\exp(-\Omega((\log N)^{1/12}))$. This bound has the same broad quasipolynomial shape as Behrend's lower bound \cite{behrend}, and improves on a long line of incremental progress culminating in the work of Bloom and Sisask \cite{bloom-sisask-barrier}. Moreover, the techniques introduced by Kelley and Meka are fundamentally different to the ones previously used to
study three-term progressions and have already seen success in applications to other problems in additive combinatorics. One such work is that of Filmus, Hatami, Hosseini and Kelman \cite{filmus-hatami-hosseini-kelman}, which establishes bounds for sets lacking so-called binary systems of linear forms. By combining the ideas of Filmus et al.\ with those of Kelley and Meka, we are able to improve the bounds for sets lacking $k$-configurations. Our main result is the following.

\begin{theorem}
\label{thm:main_result}
There is an absolute constant $C > 0$ such that the following holds. Let $\alpha \in (0,1]$ and $k \geq 2$. If $N \geq \exp(Ck^{68}\log(2/\alpha)^{16})$, then any subset $A \subseteq [N]$ of density at least $\alpha$ contains a non-degenerate $k$-configuration.
\end{theorem}

Theorem \ref{thm:main_result} can be viewed as an extension of the main result of Kelley and Meka \cite[Theorem 1.1]{kelley-meka}, which corresponds to the case $k = 2$ of our result. At the same time, it offers an improvement over \cite[Theorem 1.3]{shao}, the best bound previously available for general $k$. We remark that, if we consider separately the dependence of $F$ on $k$ and $\alpha$, then the bound in Theorem \ref{thm:main_result} is optimal up to respective exponents of $k$ and $\log(2/\alpha)$. Indeed, the classical construction of Behrend \cite{behrend} implies $F(\alpha,2) \geq \exp(\Omega(\log(2/\alpha)^2))$ and Green's bound on the clique number of random Cayley graphs \cite{green-cayley} gives $F(\alpha_0,k) \geq \exp(\Omega(k))$ for some absolute constant $\alpha_0 \in (0,1)$ (see \cite[p.\ 5]{sanders} for a more detailed discussion).

Going back to the Erdős--Moser sum-free set problem, we establish the following as an application of Theorem \ref{thm:main_result}.

\begin{theorem}
\label{thm:erdos_moser}
Let $c \in (0,\frac{1}{68})$ be arbitrary. Then for any sufficiently large finite set $A \subseteq \mathbb{Z}$, there exists a subset $B \subseteq A$ of size at least $(\log |A|)^{1+c}$ such that $b_1+b_2 \not\in A$ for any distinct $b_1,b_2 \in B$.
\end{theorem}

Theorem \ref{thm:erdos_moser} implies that $\phi(n) = \Omega((\log n)^{1+c})$ with $c = \frac{1}{69}$, so we recover a bound of the same shape as in \cite[Theorem 1.2]{sanders}. We do not claim any improvement in the value of $c$, though our constant potentially has the minor advantage of being simpler to calculate. We should note, however, that the approach to the Erdős--Moser problem via $k$-configurations is a priori limited to the range $c < 1$. Our arguments could probably be optimised and as a result one could bring the value of $c$ closer to $1$. It is not clear whether one can get arbitrarily close to $1$; one seemingly runs into similar obstacles as when attempting to improve the exponent in the Kelley--Meka bound for Roth's theorem. Since the first version of this paper appeared, further improvements in this regard have been made by Raghavan \cite{raghavan}, who proved that any subset of $[N]$ lacking three-term progressions has density at most $\exp(-\Omega((\log N)^{1/6-o(1)}))$. Incorporating the techniques of \cite{raghavan} would likely lead to an improvement in the results of the present paper; we leave this for future work. For a summary of how the bounds on $F$ influence those on $\phi$, we refer the reader to \cite[\S2]{sanders}.

The key technical input used in the proof of Theorem \ref{thm:main_result} is the graph counting lemma of Filmus et al.\ \cite[Theorem 2.1]{filmus-hatami-hosseini-kelman}. Our improvement over the work \cite{filmus-hatami-hosseini-kelman} is twofold. First, \cite{filmus-hatami-hosseini-kelman} deals with binary systems of linear forms, which are defined to be collections of linear forms in a fixed set of variables $x_1,\ldots,x_k$ such that each form depends on \emph{exactly} two variables and no two forms depend on the same pair of variables. On the other hand, each of the linear forms defining a $k$-configuration is supported on two variables, but some of them actually depend on a \emph{single} variable. At first, one might think that this requires merely a formal change to the arguments of \cite{filmus-hatami-hosseini-kelman}, for instance allowing self-loops in the graphs being counted by \cite[Theorem 2.1]{filmus-hatami-hosseini-kelman}. However, this is not the case. In a qualitative sense, we use \cite[Theorem 2.1]{filmus-hatami-hosseini-kelman} as a black box, but to overcome this difficulty, we have to incorporate some further arguments in the style of \cite{kelley-meka} and \cite{bloom-sisask}. Nevertheless, in order to obtain Theorem \ref{thm:main_result} as stated, we do have to make some improvements to the graph counting lemma of Filmus et al., but these happen purely at a quantitative level. To be specific, the dependence of the density increment parameter $\delta$ on the number of variables $k$ is not made explicit in \cite[Theorem 2.1]{filmus-hatami-hosseini-kelman}. In fact, their argument yields an exponential dependence on $k$, the primary aim of that paper being a good dependence on the density $\alpha$. However, for our application to the Erdős--Moser problem, a polynomial dependence on $k$ is paramount. Luckily, the arguments of \cite{filmus-hatami-hosseini-kelman} can be modified so as to obtain a density increment of the required strength.

The rest of the paper is organised as follows. In Section \ref{sec:graph_counting}, we establish a version of the graph counting lemma which is suitable for our application. Section \ref{sec:k_config} is devoted to the proof of Theorem \ref{thm:main_result} and occupies the main bulk of the paper. In fact, we deduce Theorem \ref{thm:main_result} as a direct consequence of Theorem \ref{thm:general_groups}, which deals with general finite abelian groups. We first give a high-level overview of the argument in the finite field setting and then proceed to transfer these arguments to the general case using the machinery of Bohr sets. In Section \ref{sec:erdos_moser}, we establish Theorem \ref{thm:erdos_moser} as a consequence of Theorem \ref{thm:main_result}. In Appendix \ref{app:aux_kelley_meka}, we collect some variants of the Kelley--Meka arguments that are needed in our proofs, but do not appear elsewhere. Finally, Appendix \ref{app:bohr_sets} contains the relevant background material on Bohr sets.

\bigskip

\noindent\textbf{Notational conventions.} In terms of notation, we mostly follow \cite{bloom-sisask}. In particular, we use Vinogradov asymptotic notation. Hence, given quantities $A$ and $B$, we write $A \ll B$ to mean $A = O(B)$, that is to say there exists an absolute constant $C > 0$ such that $|A| \leq C|B|$. This is equivalent to the notation $B \gg A$, i.e.\ $B = \Omega(A)$. We use $A \asymp B$ to denote that $A \ll B$ and $A \gg B$ hold simultaneously. 

Given a positive integer $M$, we abbreviate the set $\{1,\ldots,M\}$ to $[M]$. If $X$ is a finite non-empty set, we will use the averaging notation $\E_{x\in X}$ to denote $\frac{1}{|X|}\sum_{x\in X}$. If $X$ has the form of a $k$-fold Cartesian product $Y^k$, we may sometimes use $\E_{y_1,\ldots,y_k\in Y}$ instead of $\E_{y\in Y^k}$. Logarithmic factors are ubiquitous in the paper, so it will be convenient to use the abbreviation $\cL(\cdot)$ for the function $\log(2/\cdot)$ on the interval $(0,1]$. 

Given sets $A, B$ in an ambient abelian group $G$, we define their \emph{sumset} and \emph{difference set} by
\[A + B = \{a+b \mid a\in A,\ b \in B\}, \quad A-B = \{a-b \mid a\in A,\ b \in B\}\]
respectively. If $k$ is an integer, we write $k \cdot A = \{ka \mid a \in A\}$. This is the image of $A$ under the multiplication-by-$k$ map 
\[\psi_k \colon G \to G, \quad x \mapsto kx\] 
and should not be confused with $kA$, which is defined to be the $k$-fold sumset
\[\underbrace{A+\ldots+A}_{k\text{ times}}.\]

Unless otherwise stated, $G$ will denote a finite abelian group, written additively. By default, $G$ will be equipped with the uniform probability measure and this normalisation will be used to define $\mathbb{L}^p$-norms and inner products of complex-valued functions on $G$. Given two such functions $f$ and $g$, we also define their convolution and difference convolution to be
\[f*g \colon G \to \mathbb{C},\ x \mapsto \E_{y\in G}f(y)g(x-y), \quad\quad f \circ g \colon G \to \mathbb{C},\ x \mapsto \E_{y\in G}f(y)\overline{g(y-x)}\]
respectively. Given a further function $h \colon G \to \mathbb{C}$, we will frequently (sometimes without mention) make use of the adjoint property
\[\langle f*g,h\rangle = \langle f,h\circ g\rangle.\]
The Fourier transform plays a minor role in the paper -- it makes an appearance only in Appendix \ref{app:aux_kelley_meka}. For the sake of completeness, we define the Fourier transform of $f$ to be
\[\widehat{f} \colon \widehat{G} \to \mathbb{C},\quad \gamma \mapsto \E_{x\in G}f(x)\overline{\gamma(x)},\]
where $\widehat{G}$ denotes the group of characters of $G$, which will always carry the counting measure.

It will also be useful to consider various other probability measures on $G$. Given a function $\mu \colon G \to [0,\infty)$ such that $\lVert \mu\rVert_1 = 1$, we will identify $\mu$ with a probability measure in the obvious way. Thus, it will be convenient to make the slight abuse of terminology of calling $\mu$ a probability measure\footnote{A more correct term would be \emph{probability density function}.} and writing $\mu(A)$ for the corresponding measure of a set $A$. Hence, we have
\[\mu(A) = \E_{x\in G}1_A(x)\mu(x)\]
and the corresponding $\mathbb{L}^p$-norms and inner products are given by
\[\lVert f\rVert_{\mathbb{L}^p(\mu)} =\begin{cases}\Bigl(\E_{x\in G}|f(x)|^p\mu(x)\Bigr)^{1/p} & \text{if } p \in [1,\infty)\\\max_{x\in\text{supp}\mu}|f(x)| & \text{if } p = \infty\end{cases}, \quad \langle f,g\rangle_{\mathbb{L}^2(\mu)} = \E_{x\in G}f(x)\overline{g(x)}\mu(x).\]
Henceforth, unless specified otherwise, $\mu$ will denote the uniform probability measure on $G$. A particularly important class of probability measures arises from uniform distributions on subsets of $G$, so given a non-empty set $S \subseteq G$, we write $\mu_S = \mu(S)^{-1}1_S$ for the normalised indicator function of $S$. Thus, for example, $\mu_S(A)$ is the density of a subset $A \subseteq G$ in $S$.

\section{A quantitatively improved graph counting lemma}\label{sec:graph_counting}

The purpose of this section is to establish the following version of \cite[Theorem 2.1]{filmus-hatami-hosseini-kelman}. This is achieved by uncovering and improving various quantitative dependencies on $m$ and $\varepsilon$. In what follows, by an \emph{oriented graph} we mean a directed graph obtained by orienting the edges of a simple undirected graph. In particular, neither parallel edges nor cycles of length at most $2$ are allowed. The degree of a vertex in such a graph is defined to be its degree in the underlying undirected graph.

\begin{theorem}
\label{thm:graph_counting}
Let $\varepsilon \in (0,1]$ and let $m$ be a positive integer. Then the following holds with $\delta = \delta(\varepsilon, m) = \varepsilon^2m^{-2}/16000$. Let $\alpha \in (0,1]$ and let $H = ([k],E)$ be an oriented graph with $m$ edges such that for each $(i,j) \in E$ we have $i < j$. Let $X_1,\ldots,X_k$ be finite non-empty sets and for each $(i,j) \in E$ let $A_{i,j} \subseteq X_i \times X_j$ be a subset of density $\alpha_{i,j} \geq \alpha$. If
\[\Biggl|\E_{(x_1,\ldots,x_k)\in X_1\times\ldots\times X_k}\prod_{(i,j)\in E}1_{A_{i,j}}(x_i,x_j)-\prod_{(i,j)\in E}\alpha_{i,j}\Biggr| \geq \varepsilon\prod_{(i,j)\in E}\alpha_{i,j},\]
then there exists an edge $(i,j) \in E$ such that
\begin{itemize}
    \item[(i)] either there exist $S \subseteq X_i, T \subseteq X_j$ of densities at least $\exp(-O(\varepsilon^{-2}m^3\cL(\alpha)^2))$ such that
    \[\E_{(x,y)\in S\times T}1_{A_{i,j}}(x,y) \geq (1+\delta)\alpha_{i,j};\]
    \item[(ii)] or there exists $S \subseteq X_i$ of density at least $\exp(-O(\varepsilon^{-2}m^3\cL(\varepsilon m^{-1})\cL(\alpha)))$ such that for all $x \in S$ we have
    \[\E_{y\in X_j}1_{A_{i,j}}(x,y) \leq (1-\delta)\alpha_{i,j};\]
    \item[(iii)] or there exists $S \subseteq X_i$ of density at least $\exp(-O(\varepsilon^{-2}m^3\cL(\varepsilon m^{-1})\cL(\alpha)^2))$ such that for all $x \in S$ we have
    \[\E_{y\in X_j}1_{A_{i,j}}(x,y) \geq (1+\delta)\alpha_{i,j}.\]
\end{itemize}
\end{theorem}

\begin{remark}
\label{rem:merging_alternatives}
The case (iii) of Theorem \ref{thm:graph_counting} could be subsumed into the case (i) at the cost of a slight weakening of the uniform bound on the densities of $S$ and $T$. However, this would result in extra polylogarithmic factors in our main results, so we choose not to do so.
\end{remark}

To obtain Theorem \ref{thm:graph_counting}, we have to make an improvement to the key technical ingredient in the proof of \cite[Theorem 2.1]{filmus-hatami-hosseini-kelman}, namely \cite[Lemma 2.8]{filmus-hatami-hosseini-kelman}. We accomplish this by closely following the proof of \cite[Lemma 2.8]{filmus-hatami-hosseini-kelman} and making careful choices of parameters where necessary. Before doing so, we recall some notation from \cite[\S2]{filmus-hatami-hosseini-kelman}. For a function $f \colon X \to \mathbb{C}$ on a finite non-empty set $X$ and a parameter $p \in [1,\infty)$, we define the $\mathbb{L}^p$-norm of $f$ to be
\[\lVert f\rVert_p \vcentcolon= \Bigl(\E_{x\in X}|f(x)|^p\Bigr)^{1/p}.\]
In a similar vein, if $f \colon X \times Y \to \mathbb{R}$ is a function on the Cartesian product of finite non-empty sets $X,Y$ and $p, q$ are positive integers, we define the corresponding \emph{grid-semi-norm} of $f$ to be
\[\lVert f\rVert_{U(p,q)} \vcentcolon= \Biggl|\E_{x\in X^p,y\in Y^q}\prod_{i=1}^{p}\prod_{j=1}^{q}f(x_i,y_j)\Biggr|^{\frac{1}{pq}}.\]
All relevant properties of grid-semi-norms can be found in \cite[\S2]{filmus-hatami-hosseini-kelman}. For an oriented graph $H$, define $\kappa(H) \vcentcolon= 2|E(H)| - d_1(H)$, where $d_1(H)$ is the number of vertices of degree $1$ in $H$. Observe that $\kappa(H)$ is equal to the sum of the degrees of all vertices of degree greater than $1$ in $H$. In particular, if all degrees in $H$ are at most $1$, then $\kappa(H) = 0$, and otherwise $\kappa(H) \geq 2$.

\begin{lemma}
\label{lm:main_technical_lemma}
Let $\alpha, \varepsilon \in (0,1]$ and let $H$ be an acyclic oriented graph with $m$ edges. Let $(X_v)_{v\in V(H)}$ be finite non-empty sets and for each $(u, v) \in E(H)$ let $f_{u,v} \colon X_u \times X_v \to \{0,1\}$ be a function with mean $\alpha_{u,v} \geq \alpha$. If
\[\Biggl|\E_{x\in\prod_{v\in V(H)}X_v}\prod_{(u,v)\in E(H)}f_{u,v}(x_u,x_v)-\prod_{(u,v)\in E(H)}\alpha_{u,v}\Biggr| \geq \varepsilon\prod_{(u,v)\in E(H)}\alpha_{u,v},\]
then writing $\widetilde{\delta} = \varepsilon^2\kappa(H)^{-2}/1000$, there exist an edge $(u, v) \in E(H)$ and a positive integer $p$ such that
\begin{itemize}
    \item[(i)] either there exists a positive integer $r$ such that $pr \leq 8\widetilde{\delta}^{-1}m\cL(\alpha)$ and $\lVert f_{u,v}\rVert_{U(r,p)} \geq (1+\widetilde{\delta})\alpha_{u,v}$;
    \item[(ii)] or $p \leq 8\widetilde{\delta}^{-1}m\cL(\alpha)$ and $\lVert \E_{y\in X_v}f_{u,v}(\cdot,y)-\alpha_{u,v}\rVert_p \geq \widetilde{\delta}\alpha_{u,v}$.
\end{itemize}
\end{lemma}

\begin{proof}
    We will be fairly brief since the argument overlaps heavily with that of \cite{filmus-hatami-hosseini-kelman}; we indicate the key changes that have to be made. We proceed by induction on $\kappa(H)$. As in \cite{filmus-hatami-hosseini-kelman}, we may reduce to the case when $H$ has no isolated vertices/edges, so in particular $\kappa(H) \geq 2$. For each $(u,v) \in E(H)$, consider the normalised variant of $f_{u,v}$ given by $F_{u,v} \vcentcolon= \alpha_{u,v}^{-1}f_{u,v}$. Since $H$ is acyclic, it contains a vertex of outdegree $0$, call it $s$. Let $u_1,\ldots,u_{\ell}$ be the in-neighbours of $s$, where $\ell \geq 1$. Let $H'$ be the graph obtained by removing $s$ from $H$. Since $H$ has no isolated edges, it follows that $d_1(H') \geq d_1(H)-\ell$ and hence
    \[\kappa(H') = 2|E(H')| - d_1(H') \leq 2(|E(H)|-\ell) - (d_1(H)-\ell) = \kappa(H)-\ell.\]
    Thus, we can make the assumption that
    \begin{equation}\label{eq:ind_hyp}
        \Biggl|\E_{x\in\prod_{v\in V(H')}X_v}\prod_{(u,v)\in E(H')}F_{u,v}(x_u,x_v)-1\Biggr|\leq\varepsilon',
    \end{equation}
    where we define\footnote{Note the change in the choice of $\varepsilon'$ here -- in \cite{filmus-hatami-hosseini-kelman}, the choice $\varepsilon' = \varepsilon/2$ was made.}
    \[\varepsilon' \vcentcolon= \frac{\kappa(H)-1}{\kappa(H)}\varepsilon.\]
    Indeed, if \eqref{eq:ind_hyp} doesn't hold, then we may finish by applying the induction hypothesis to $H'$. In particular, we may certainly assume that
    \[\E_{x\in\prod_{v\in V(H')}X_v}\prod_{(u,v)\in E(H')}F_{u,v}(x_u,x_v) \leq 2.\]
    If we let $p \vcentcolon= 4\lceil m\cL(\alpha)\rceil$ and $q$ be such that $\frac{1}{p}+\frac{1}{q} = 1$, then we may use the above as in \cite{filmus-hatami-hosseini-kelman} to bound
    \begin{equation}\label{eq:bounded_by_three}
        \Biggl(\E_{x\in\prod_{v\in V(H')}X_v}\prod_{(u,v)\in E(H')}F_{u,v}(x_u,x_v)^q\Biggr)^{1/q} \leq 3.
    \end{equation}
    
    We now turn to the main part of the argument. Consider first the case when $\ell = 1$. Then \eqref{eq:ind_hyp} implies
    \[\Biggl|\E_{x\in\prod_{v\in V(H)}X_v}\prod_{(u,v)\in E(H)}F_{u,v}(x_u,x_v)-\E_{x\in\prod_{v\in V(H')}X_v}\prod_{(u,v)\in E(H')}F_{u,v}(x_u,x_v)\Biggr| \geq \varepsilon-\varepsilon',\]
    which can be rewritten as
    \[\Biggl|\E_{x\in\prod_{v\in V(H')}X_v}\prod_{(u,v)\in E(H')}F_{u,v}(x_u,x_v)\Bigl(\E_{x_s\in X_s}F_{u_1,s}(x_{u_1},x_s)-1\Bigr)\Biggr| \geq \frac{\varepsilon}{\kappa(H)}.\]
    Using Hölder's inequality and the estimate \eqref{eq:bounded_by_three} as in \cite{filmus-hatami-hosseini-kelman}, we obtain that
    \[\Bigl(\E_{x\in X_{u_1}}\Bigl(\E_{y\in X_s}F_{u_1,s}(x,y)-1\Bigr)^p\Bigr)^{1/p} \geq \frac{\varepsilon}{3\kappa(H)} > \widetilde{\delta},\]
    which means that (ii) holds.
    
    Suppose now that $\ell \geq 2$. Let $H''$ be the graph obtained from $H$ by adding a vertex $s'$, removing the edge $(u_1,s)$ and adding the edge $(u_1,s')$; we also let $X_{s'} \vcentcolon= X_s$ and $f_{u_1,s'} \vcentcolon= f_{u_1,s}$. Then note that $d_1(H'') \geq d_1(H)+1$, so we have
    \[\kappa(H'') = 2|E(H'')| - d_1(H'') \leq 2|E(H)| - (d_1(H)+1) = \kappa(H)-1.\]
    Thus, we may assume that
    \[\Biggl|\E_{x\in\prod_{v\in V(H'')}X_v}\prod_{(u,v)\in E(H'')}F_{u,v}(x_u,x_v)-1\Biggr|\leq\varepsilon'\]
    as otherwise we are done by the induction hypothesis applied to $H''$. By the triangle inequality, it follows that
    \[\Biggl|\E_{x\in\prod_{v\in V(H'')}X_v}\prod_{(u,v)\in E(H'')}F_{u,v}(x_u,x_v)-\E_{x\in\prod_{v\in V(H)}X_v}\prod_{(u,v)\in E(H)}F_{u,v}(x_u,x_v)\Biggr| \geq \varepsilon-\varepsilon'.\]
    This can be rewritten as
    \[\Biggl|\E_{x\in\prod_{v\in V(H')}X_v}\prod_{(u,v)\in E(H')}F_{u,v}(x_u,x_v)\E_{x_s\in X_s}J(x_{u_1},x_s)\prod_{i=2}^{\ell}F_{u_i,s}(x_{u_i},x_s)\Biggr| \geq \frac{\varepsilon}{\kappa(H)},\]
    where we define
    \[J \colon X_{u_1}\times X_s \to \mathbb{R}, \quad (x,y) \mapsto F_{u_1,s}(x,y) - \E_{y'\in X_s}F_{u_1,s}(x,y').\]
    We may now use Hölder's inequality, the Cauchy--Schwarz inequality, the Gowers--Cauchy--Schwarz inequality \cite[Lemma 2.5]{filmus-hatami-hosseini-kelman} and the estimate \eqref{eq:bounded_by_three} in the same way as in \cite{filmus-hatami-hosseini-kelman} to bound the left-hand side. Hence, we obtain
    \[\lVert J\rVert_{U(2,p)}\prod_{i=2}^{\ell}\lVert F_{u_i,s}\rVert_{U(2(\ell-1),p)} \geq \frac{\varepsilon}{3\kappa(H)}.\]
    If there exists $i \in \{2,\ldots,\ell\}$ such that $\lVert F_{u_i,s}\rVert_{U(2(\ell-1),p)} \geq 1+\widetilde{\delta}$, then (i) holds with $r = 2(\ell-1)$. Thus, we may assume that $\lVert F_{u_i,s}\rVert_{U(2(\ell-1),p)} \leq 1+\widetilde{\delta}$ for all $i \in \{2,\ldots,\ell\}$, so $\lVert J\rVert_{U(2,p)} \geq \eta$, where
    \[\eta \vcentcolon= \frac{\varepsilon}{3\kappa(H)(1+\widetilde{\delta})^{\ell-1}} \geq \frac{\varepsilon}{10\kappa(H)}.\]
    The last inequality holds since\footnote{Note that in \cite{filmus-hatami-hosseini-kelman}, the weaker estimate $(1+\widetilde{\delta})^{\ell-1} \leq 2^{\ell-1}$ is applied.}
    \[(1+\widetilde{\delta})^{\ell-1} \leq \exp(\widetilde{\delta}(\ell-1)) \leq \exp\Bigl(\frac{\varepsilon^2}{1000\kappa(H)^2}\cdot\kappa(H)\Bigr) < \exp(1/1000) < 3.\]
    Applying \cite[Lemma 2.9]{filmus-hatami-hosseini-kelman}, we obtain that $\lVert 1+J\rVert_{U(2,p')} \geq 1+\eta^2/5$ with $p' \vcentcolon= 2\lceil p/\eta^2\rceil$. If we now let $D \vcentcolon= F_{u_1,s}-J$, then the triangle inequality implies that
    \[\lVert F_{u_1,s}\rVert_{U(2,p')} + \lVert D-1\rVert_{U(2,p')}\geq \lVert 1+J\rVert_{U(2,p')} \geq 1+2\widetilde{\delta},\]
    so either $\lVert F_{u_1,s}\rVert_{U(2,p')} \geq 1+\widetilde{\delta}$ or $\lVert D-1\rVert_{U(2,p')} \geq \widetilde{\delta}$. In the former case, it follows that (i) holds with $r = 2$ and $p'$ playing the role of $p$, so we are done. On the other hand, since $D(x,y)$ doesn't depend on $y$, the latter case degenerates to
    \[\Bigl\lVert \E_{y\in X_s}F_{u_1,s}(\cdot,y)-1\Bigr\rVert_{p'} \geq \widetilde{\delta},\]
    so (ii) holds with $p'$ in place of $p$. This concludes the proof.
\end{proof}

With the ancillary results of \cite[\S2]{filmus-hatami-hosseini-kelman} at hand, it is now a short step from Lemma \ref{lm:main_technical_lemma} to Theorem \ref{thm:graph_counting} -- we mimic the deduction of \cite[Theorem 2.1]{filmus-hatami-hosseini-kelman} from \cite[Lemma 2.8]{filmus-hatami-hosseini-kelman}.

\begin{proof}[Proof of Theorem \ref{thm:graph_counting}]
    We apply Lemma \ref{lm:main_technical_lemma} with $f_{i,j} \vcentcolon= 1_{A_{i,j}}$ for $(i,j) \in E$. If (i) holds, then by \cite[Corollary 2.3]{filmus-hatami-hosseini-kelman} applied with $K_{r,p}$ and $1_{A_{u,v}}$ in place of $H$ and $A$ respectively, we arrive at the conclusion (i) of Theorem \ref{thm:graph_counting}. Otherwise, if (ii) holds, then we apply \cite[Lemma 2.7]{filmus-hatami-hosseini-kelman} to the function
    \[f \colon X_u \to [0,1], \quad x \mapsto \E_{y\in X_v}1_{A_{u,v}}(x,y)\]
    to conclude that
    \begin{itemize}
        \item[(a)] either $S \vcentcolon= \{x \in X_u \mid f(x) \geq (1+\widetilde{\delta}/4)\alpha_{u,v}\} \subseteq X_u$ has density at least 
        \[(\widetilde{\delta}\alpha_{u,v})^p/4 \geq \exp(-O(\varepsilon^{-2}m^3\cL(\varepsilon m^{-1})\cL(\alpha)^2));\]
        \item[(b)] or $S \vcentcolon= \{x \in X_u \mid f(x) \leq (1-\widetilde{\delta}/4)\alpha_{u,v}\} \subseteq X_u$ has density at least 
        \[\widetilde{\delta}^p/4 \geq \exp(-O(\varepsilon^{-2}m^3\cL(\varepsilon m^{-1})\cL(\alpha))).\]
    \end{itemize}
    In case (a), we see that the conclusion (iii) of Theorem \ref{thm:graph_counting} holds, whereas in case (b), it follows that the conclusion (ii) of Theorem \ref{thm:graph_counting} holds.
\end{proof}

\section{Kelley--Meka bounds for sets free of \texorpdfstring{$k$}{k}-configurations}\label{sec:k_config}

In this section, we establish Theorem \ref{thm:main_result}. As mentioned in Section \ref{sec:intro}, our arguments naturally yield the following more general variant. It should be compared to \cite[Theorem 1.5]{filmus-hatami-hosseini-kelman}, which establishes an analogous bound for binary systems of linear forms. Our argument should likewise allow for more general coefficients in Theorem \ref{thm:general_groups}. However, since this is not our main goal and the paper is already quite technical in its current form, we do not pursue this here.

\begin{theorem}
\label{thm:general_groups}
Let $G$ be a finite abelian group of odd order and let $k \geq 2$ be an integer. Let $A \subseteq G$ be a subset of density $\alpha > 0$. Then
\begin{equation}\label{eq:k_config_count}
    \mathbb{P}_{x_1,\ldots,x_k\in G}\Bigl(\frac{x_i+x_j}{2} \in A \text{ for all } 1 \leq i \leq j \leq k\Bigr) \geq \exp(-O(k^{68}\cL(\alpha)^{16})).
\end{equation}
In particular, if $A$ contains no non-degenerate $k$-configuration, then 
\begin{equation}\label{eq:density_bound}
    |G| \leq \exp(O(k^{68}\cL(\alpha)^{16})).
\end{equation}
\end{theorem}

We begin by giving a quick deduction of Theorem \ref{thm:main_result} from Theorem \ref{thm:general_groups}. This is entirely standard: we embed the interval $[N]$ into a suitably large cyclic group of odd order and apply the bound \eqref{eq:density_bound}.

\begin{proof}[Proof of Theorem \ref{thm:main_result} assuming Theorem \ref{thm:general_groups}]
    Assume to the contrary that $A$ doesn't contain any non-degenerate $k$-configuration. Set $N' \vcentcolon= 2N+1$ and let $\pi \colon \mathbb{Z} \to \mathbb{Z}/N'\mathbb{Z}$ be the natural projection. Then $A' \vcentcolon= \pi(A)$ contains no non-degenerate $k$-configuration. Indeed, if this were not the case, there would exist distinct $x_1, \ldots, x_k \in A$ such that for all $1 \leq i < j \leq k$ we have $(\pi(x_i)+\pi(x_j))/2 \in A'$. But this means that for any such $i,j$ there is some $z \in A$ such that $x_i+x_j \equiv 2z \pmod {N'}$. Since $x_i,x_j,z\in [N]$, we must have $x_i+x_j = 2z$, i.e.\ $z = (x_i+x_j)/2$. But then $A$ contains a non-degenerate $k$-configuration, which contradicts our starting assumption. Therefore, by Theorem \ref{thm:general_groups}, we have
    \[N' \leq \exp(O(k^{68}\cL(|A'|/N')^{16})).\]
    Since $N \leq N'$ and $|A'|/N' \geq |A|/(3N) = \alpha/3$, we obtain that
    \[N \leq \exp(O(k^{68}\cL(\alpha)^{16})).\]
    Taking $C$ to be sufficiently large now gives the desired contradiction.
\end{proof}

Thus, our main task for this section is to prove Theorem \ref{thm:general_groups}. Since the details are rather technical, we first give a sketch of a simplified version of the argument in the hope of illuminating the main ideas. Along the way, we also include a comparison with the work of Filmus et al.\ \cite{filmus-hatami-hosseini-kelman}, highlighting the extra difficulties that arise in our case. For the purposes of this discussion, we work in the setting of vector spaces over finite fields. This allows one to express the main ideas much more cleanly, without having to go through all the technicalities of the general case.

\bigskip

\noindent\textbf{Outline of the argument.} Suppose that $G = \mathbb{F}_q^n$, where $q$ is an odd prime. As is common with problems of this kind, we employ a density increment approach. The basic idea is that if the number of $k$-configurations in $A$ deviates from what one would expect for a random set of the same density, then some translate of $A$ has increased density on a large subspace of $G$. One then iterates this argument; since the density of $A$ is bounded above by $1$, this process terminates in a bounded number of steps, thereby giving the conclusion. The starting point, therefore, is to apply Theorem \ref{thm:graph_counting} (though for the level of the discussion here the symmetric variant \cite[Theorem 1.4]{filmus-hatami-hosseini-kelman} would also suffice). We take $H$ to be the transitive orientation of the complete graph $K_k$, the sets $X_1, \ldots, X_k$ to be copies of $A$ and each set $A_{i,j}$ to consist of the pairs $(x,y) \in A \times A$ such that $x+y \in 2\cdot A$. 

There is a serious obstacle, however, to the application of the graph counting lemma, which is that we a priori have no control over the density of the sets $A_{i,j}$. Indeed, in the extreme case when $A$ is free of three-term arithmetic progressions, each $A_{i,j}$ is empty! This should be compared with the situation in \cite{filmus-hatami-hosseini-kelman}, where such difficulties are absent. More concretely, in the case of the binary system given by $L_{i,j}(x_1,\ldots,x_k) = x_i+x_j$ for $1 \leq i < j \leq k$, one instead takes each of $X_1,\ldots,X_k$ to be the whole of $G$ and $A_{i,j}$ to consist of the pairs $(x,y) \in G\times G$ such that $x+y \in A$. Observe that, in this case, the bipartite graph represented by $A_{i,j}$ has density $\alpha$. Moreover, it is regular -- every vertex has degree exactly equal to $|A|$. Even though our situation is more complicated, there is still a way out. Indeed, the case when there is a significant discrepancy in the density of $A_{i,j}$ from $\alpha$ is precisely the starting assumption in the Kelley--Meka proof of Roth's theorem. Therefore, this case can be shown to lead to a density increment of the desired kind.

In the complementary case, the graph counting lemma provides us with suitably dense sets $S, T \subseteq A$ such that either $A_{i,j}$ has significantly increased density on $S \times T$ or the degree in $A_{i,j}$ of each vertex from $S$ is significantly lower than average. The process of deducing a suitable density increment from the former is precisely the main concern of \cite[\S3]{filmus-hatami-hosseini-kelman}. Briefly, this can be achieved as follows. One rewrites the conclusion as
\[\langle \mu_S*\mu_T, 1_{2\cdot A}\rangle \geq (1+\delta)\alpha\]
and applies the adjoint property of convolutions to transform this into
\[\langle \mu_{2\cdot A}\circ\mu_T, \mu_S\rangle \geq 1+\delta.\]
From there, an application of Hölder's inequality implies an unusually large $\mathbb{L}^p$-norm of the difference convolution $\mu_{2\cdot A}\circ\mu_T$ for a not too large value of the parameter $p$. An asymmetric version of the dependent random choice argument of Kelley and Meka followed by almost-periodicity then yields a density increment of $2\cdot A$ on a subspace of suitably small codimension. Rescaling, we obtain an adequate density increment of $A$. We remark that, out of the main components of the Kelley--Meka proof of Roth's theorem, we did not have to use spectral non-negativity since we already had an upwards deviation from an expected count. This is because spectral non-negativity (more precisely an analogue thereof, namely \cite[Lemma 2.9]{filmus-hatami-hosseini-kelman}) has already been used in the argument leading to this alternative of the graph counting lemma.

In the remaining case, $1_{2\cdot A}\circ\mu_A$ is bounded above by $(1-\delta)\alpha$ everywhere on $S$, whence averaging yields
\[\langle \mu_{2\cdot A} \circ \mu_A, \mu_S\rangle \leq 1-\delta.\]
We note in passing that, in the context of \cite{filmus-hatami-hosseini-kelman}, this cannot happen since the graph $A_{i,j}$ is regular. In our case, however, this means that
\[|\langle (\mu_{2\cdot A}-1)\circ (\mu_A-1), \mu_S\rangle| \geq \delta,\]
whence an application of Hölder's inequality followed by the decoupling inequality of Kelley and Meka \cite[Proposition 2.12]{kelley-meka} implies that at least one of
\[\lVert (\mu_{2\cdot A}-1)\circ (\mu_{2\cdot A}-1)\rVert_p,\quad \lVert (\mu_A-1)\circ (\mu_A-1)\rVert_p\]
must be large for a not too large value of the parameter $p$. From this, one can use spectral non-negativity together with the other Kelley--Meka techniques (sifting, almost-periodicity) to obtain a density increment of either $2\cdot A$ or $A$. By dilating if necessary, these two cases can be unified into a density increment of $A$.

\bigskip

We now begin the process of translating the above ideas into the framework of Bohr sets. In general finite abelian groups, Bohr sets serve as a substitute for subspaces. This complicates the arguments on a technical level, but the core ideas remain essentially the same. The reader unfamiliar with this topic may wish to consult Appendix \ref{app:bohr_sets} before reading the rest of this section.

The proof of Theorem \ref{thm:general_groups} is based on the following density increment result, which plays a role analogous to that of \cite[Theorem 3.3]{filmus-hatami-hosseini-kelman}.

\begin{proposition}
\label{prop:density_increment}
Let $G$ be a finite abelian group of odd order and let $k \geq 2$ be an integer. Let $B = \mathrm{Bohr}(\Gamma;\rho)$ be a regular Bohr set of rank $d$ in $G$ and let $A \subseteq B$ be a subset of density $\alpha > 0$. Then
\begin{itemize}
    \item[(i)] either the proportion of $k$-tuples $(x_1,\ldots,x_k)\in G^k$ such that $\frac{x_i+x_j}{2} \in A$ for all $1 \leq i \leq j \leq k$ is at least $\exp(-O((k^2\log k)d\cL(\alpha/d)))\rho^{2kd}$;
    \item[(ii)] or there exist $\Delta \subseteq \widehat{G}$ of size at most $O(k^{46}\cL(\alpha)^{11})$,  $\rho' \geq \rho\exp(-O(k\cL(\alpha/d)+k^6\cL(\alpha)^2))$ and $x \in G$ such that $B' = \mathrm{Bohr}(\Gamma';\rho')$ is a regular Bohr set and
    \[\mu_{B'}(A-x) \geq (1+\Omega(k^{-5}))\alpha,\]
    where $\Gamma' = \Gamma \cup \{\gamma \psi_2^{\sigma} \mid \gamma \in \Gamma\} \cup\Delta$ for some $\sigma \in \{\pm1\}$.
\end{itemize}
\end{proposition}

Theorem \ref{thm:general_groups} follows from Proposition \ref{prop:density_increment} by iteration. For the most part, this follows standard lines. Nevertheless, we have to make an additional observation about the behaviour of the rank of the ambient Bohr set. Indeed, per (ii) of Proposition \ref{prop:density_increment}, it may seem at first glance that at each step of the iteration, the frequency set doubles in size. This would cause the rank to grow exponentially, which would result in poor bounds in Theorem \ref{thm:general_groups}. Luckily, by inspecting the specific structure of the frequency set, one readily sees that the rank in fact grows quadratically. This kind of observation seems to have first appeared in the work of Pilatte \cite{pilatte} (see \cite[Remark 3.4]{pilatte} for a more detailed discussion).

\begin{proof}[Proof of Theorem \ref{thm:general_groups} assuming Proposition \ref{prop:density_increment}]
    Consider pairs of sequences $(A_s)_{0\leq s\leq t}$, $(B_s)_{0\leq s\leq t}$ such that the following hold for $0 \leq s \leq t$:
    \begin{itemize}
        \item[(i)] $A_0 = A$, $B_0 = G$;
        \item[(ii)] $B_s = \mathrm{Bohr}(\Gamma_s;\rho_s)$ is a regular Bohr set of rank $d_s$, where $\Gamma_s \subseteq \widehat{G}$, $\rho_s \in [0,2]$;
        \item[(iii)] $A_s \subseteq B_s$ is a subset of density $\alpha_s$;
        \item[(iv)] if $s \geq 1$, then $A_s = (A_{s-1}-x_s) \cap B_s$ for some $x_s \in G$;
        \item[(v)] if $s \geq 1$, then $\alpha_s \geq (1+ck^{-5})\alpha_{s-1}$, where $c > 0$ is an absolute constant;
        \item[(vi)] if $s \geq 1$, then $\Gamma_s = \Gamma_{s-1} \cup \{\gamma \psi_2^{\sigma_s} \mid \gamma \in \Gamma_{s-1}\} \cup \Delta_s$ for some $\sigma_s \in \{\pm1\}$ and $\Delta_s \subseteq \widehat{G}$ of size at most $O(k^{46}\cL(\alpha_{s-1})^{11})$;
        \item[(vii)] if $s \geq 1$, then $\rho_s \geq \rho_{s-1}\exp(-O(k\cL(\alpha_{s-1}/d_{s-1})+k^6\cL(\alpha_{s-1})^2))$.
    \end{itemize}
    In particular, the sequences $(\alpha_s)_{0\leq s\leq t}$ and $(d_s)_{0\leq s\leq t}$ are increasing. Moreover, property (v) implies via induction that $\alpha_s \geq (1+ck^{-5})^s\alpha$ for all $0 \leq s \leq t$. Since $\alpha_t \leq 1$, it follows that $t \ll k^5\cL(\alpha)$. In particular, we may choose such a pair of sequences whose length $t+1$ is maximal. We now estimate the rank and width of $B_t$. First, using (vi), it is straightforward to show by induction that for all $0 \leq s \leq t$ we have the inclusion
    \[\Gamma_s \subseteq \bigcup_{r=0}^{s}\{\gamma \psi_2^a \mid \gamma \in \Delta_r,\ a \in \mathbb{Z},\ |a| \leq s-r\},\]
    where we separately define $\Delta_0 \vcentcolon= \{1\}$. Hence, by employing the union bound, we get
    \[d_t = |\Gamma_t| \leq \sum_{r=0}^{t}(2(t-r)+1)|\Delta_r| \ll t^2k^{46}\cL(\alpha)^{11} \ll k^{56}\cL(\alpha)^{13}.\]
    Furthermore, (vii) implies that
    \begin{align*}
        \rho_t &\geq \prod_{s=1}^{t}\exp(-O(k\cL(\alpha_{s-1}/d_{s-1})+k^6\cL(\alpha_{s-1})^2))\\
        &\geq \exp(-O(k\cL(\alpha/d_t)+k^6\cL(\alpha)^2))^t\\
        &\geq \exp(-O(k^{11}\cL(\alpha)^3)).
    \end{align*}
    By maximality of $t$, upon applying Proposition \ref{prop:density_increment} to $A_t \subseteq B_t$, we must end up in case (i). But this means that the proportion of $k$-tuples $(x_1,\ldots,x_k) \in G^k$ generating a $k$-configuration in $A_t$ is at least
    \begin{align*}
        \exp(-O((k^2\log k)d_t\cL(\alpha_t/d_t)))\rho_t^{2kd_t} &\geq \exp(-O(kd_t \cdot k^{11}\cL(\alpha)^3))\\
        &= \exp(-O(k\cdot k^{56}\cL(\alpha)^{13}\cdot k^{11}\cL(\alpha)^3))\\
        &= \exp(-O(k^{68}\cL(\alpha)^{16})).
    \end{align*}
    To conclude, note that by inductively applying (iv), it follows that $A_t$ is a subset of a translate of $A$. Since $k$-configurations are translation-invariant, we thus obtain the claimed bound \eqref{eq:k_config_count}. If $A$ contains no non-degenerate $k$-configuration, then the left-hand side in \eqref{eq:k_config_count} can be crudely upper bounded by
    \[\mathbb{P}_{x_1,\ldots,x_k\in G}(x_i = x_j \text{ for some distinct } i,j \in [k]) \leq \binom{k}{2}/|G|,\]
    from which \eqref{eq:density_bound} follows by rearranging.
\end{proof}

The goal for the rest of this section is to establish Proposition \ref{prop:density_increment}. To accomplish this, we will require some preparation. Specifically, we will have to adapt several intermediate results of \cite{bloom-sisask} and \cite{kelley-meka} to our setting. The first result says that, in a suitable local setting, self-regularity implies mixing for three-variable linear equations. Here, we informally say that a set $A$ is \emph{self-regular} if the $\mathbb{L}^p$-norm of the difference convolution of $A$ with itself is roughly equal to the `expected' value (for a precise definition, see \cite[Definition 2.3]{kelley-meka}). Hence, our result can be viewed as a local variant of \cite[Theorem 2.11]{kelley-meka}. The proof consists of carrying out the Hölder lifting and spectral non-negativity/unbalancing steps of the Kelley--Meka proof as exposited by Bloom and Sisask \cite{bloom-sisask}. It also requires an intermediate application of the local decoupling inequality of Kelley and Meka \cite[Lemma 5.8]{kelley-meka}.

\begin{proposition}
\label{prop:few_solutions_to_lp_norm}
There is an absolute constant $c > 0$ such that the following holds. Let $B, B' \subseteq G$ be regular Bohr sets of rank $d$. Let $A_1,A_2 \subseteq B$ be subsets of densities at least $\alpha > 0$ and let $C \subseteq B'$ be a subset of density $\gamma > 0$. Suppose that $B' \subseteq B_{c\varepsilon\alpha/d}$ for some $\varepsilon \in (0,1]$ and let $B'', B''' \subseteq B_{c/d}'$ be non-empty symmetric subsets. If
\[|\langle \mu_{A_1}\circ\mu_{A_2}, \mu_C\rangle - \mu(B)^{-1}| \geq \varepsilon\mu(B)^{-1},\]
then there exist $j \in \{1,2\}$, a positive integer $p \ll \varepsilon^{-1}\cL(\gamma)$ and $x \in G$ such that
\[\lVert \mu_{A_j}\circ\mu_{A_j}\rVert_{\mathbb{L}^p(\mu_{B''}*\mu_{B'''+x})} \geq \Bigl(1+\frac{\varepsilon}{16}\Bigr)\mu(B)^{-1}.\]
\end{proposition}
\begin{proof}
    Since $B$ is symmetric and the conclusion is invariant under replacing $A_2$ by ${-A_2}$, by considering ${-A_2} \subseteq B$ in place of $A_2$, we may instead assume that
    \[|\langle \mu_{A_1}*\mu_{A_2}, \mu_C\rangle - \mu(B)^{-1}| \geq \varepsilon\mu(B)^{-1}.\]
    Hence, by Proposition \ref{prop:holder_lifting}, there exists an even positive integer $p \ll \cL(\gamma)$ such that
    \[\lVert (\mu_{A_1}-\mu_B)*(\mu_{A_2}-\mu_B)\rVert_{\mathbb{L}^p(\mu_{B'})} \geq \frac{1}{2}\varepsilon\mu(B)^{-1}.\]
    Then consider the probability measures
    \[\nu' \vcentcolon= \mu_{B''} * \mu_{B'''}, \quad \nu \vcentcolon= \nu' * \nu'.\]
    Since $\nu$ is supported on $4B_{c/d}'$, it follows by Lemma \ref{lm:domination_by_convolution} that $\mu_{B'} \leq 2\mu_{B_{1+4c/d}'}*\nu$. Furthermore, since $\nu'$ is symmetric, we have $\widehat{\nu} \geq 0$. Hence, Lemma \ref{lm:averaging_argument} and \cite[Lemma 5.8]{kelley-meka} imply that
    \[\max(\lVert f\circ f\rVert_{\mathbb{L}^p(\nu)},\lVert g\circ g\rVert_{\mathbb{L}^p(\nu)}) \geq \frac{1}{4}\varepsilon\mu(B)^{-1},\]
    where we write $f \vcentcolon= \mu_{A_1} - \mu_B$, $g \vcentcolon= \mu_{A_2} - \mu_B$. Thus, by \cite[Proposition 18]{bloom-sisask} and the remarks following \cite[Lemma 7]{bloom-sisask}, there exist $j \in \{1,2\}$ and a positive integer 
    \[p' \ll \varepsilon^{-1}p \ll \varepsilon^{-1}\cL(\gamma)\] 
    such that
    \[\lVert \mu_{A_j} \circ \mu_{A_j}\rVert_{\mathbb{L}^{p'}(\nu)} \geq \Bigl(1+\frac{\varepsilon}{16}\Bigr)\mu(B)^{-1}.\]
    The desired conclusion now follows by applying Lemma \ref{lm:averaging_argument}.
\end{proof}

We next establish a bespoke variant of \cite[Proposition 15]{bloom-sisask}. The main difference is that we allow the sets $A_1,A_2$ to be different and we make explicit the frequency sets of our Bohr sets as well as the dependencies of all bounds on the parameter $\varepsilon$. We should stress that these differences are mainly cosmetic in nature and the proof encompasses the same techniques: dependent random choice and almost-periodicity. The treatment of these techniques is deferred to Appendix \ref{app:aux_kelley_meka} and entails several optimisations of the arguments of \cite{bloom-sisask}.

\begin{proposition}
\label{prop:lp_norm_to_density_increment}
There are absolute constants $c, C > 0$ such that the following holds. Let $B^{(1)}$ and $B^{(2)} = \mathrm{Bohr}(\Gamma;\rho)$ be regular Bohr sets of rank $d$ in $G$ such that $B^{(2)} \subseteq B_{c/d}^{(1)}$. Let $A_1, A_2$ be subsets of translates of a non-empty set $B \subseteq G$ with respective densities $\alpha_1,\alpha_2 > 0$. Suppose that
\[\lVert \mu_{A_1} \circ \mu_{A_2}\rVert_{\mathbb{L}^p(\mu_{B^{(1)}}*\mu_{B^{(2)}+x})} \geq (1+\varepsilon)\mu(B)^{-1}\]
for some $x \in G$, $\varepsilon \in (0,1]$ and integer $p \geq C\varepsilon^{-1}\cL(\varepsilon)$. Then there exist $\Delta \subseteq \widehat{G}$ of size $d'$ and $\rho'$ such that 
\[d' \ll \varepsilon^{-2}p^2\cL(\varepsilon(\alpha_1\alpha_2)^{1/2})^2\cL(\alpha_1)\cL(\alpha_2), \quad \rho' \gg \rho\varepsilon(\alpha_1\alpha_2)^{1/2}/(d^3d')\] 
and $B' = \mathrm{Bohr}(\Gamma\cup\Delta;\rho')$ is a regular Bohr set with the property that
\[\lVert \mu_{A_1}*\mu_{B'}\rVert_{\infty} \geq (1+\varepsilon/2)\mu(B)^{-1}.\]
\end{proposition}
\begin{proof}
    To begin, we may assume for convenience that $x = 0$. Indeed, the general case then follows on observing that
    \[\lVert \mu_{A_1} \circ \mu_{A_2}\rVert_{\mathbb{L}^p(\mu_{B^{(1)}}*\mu_{B^{(2)}+x})} = \lVert \mu_{A_1} \circ \mu_{A_2+x}\rVert_{\mathbb{L}^p(\mu_{B^{(1)}}*\mu_{B^{(2)}})}\]
    and applying the special case with $A_2+x$ instead of $A_2$. Hence, by Lemma \ref{lm:asymmetric_sifting}, we obtain subsets $A_1' \subseteq B^{(1)}$, $A_2' \subseteq B^{(2)}$ of respective densities $\alpha_1' \gg \alpha_1^p$, $\alpha_2' \gg \alpha_2^p$ such that
    \[\langle \mu_{A_1'}\circ\mu_{A_2'}, 1_S\rangle \geq 1-\frac{\varepsilon}{16},\]
    where we write
    \[S \vcentcolon= \{x \in G \mid (\mu_{A_1}\circ\mu_{A_2})(x) \geq (1-\varepsilon/8)\lVert \mu_{A_1}\circ\mu_{A_2}\rVert_{\mathbb{L}^p(\mu_{B^{(1)}}*\mu_{B^{(2)}})}\}.\]
    Letting $S' \vcentcolon= (A_1'-A_2') \cap S$, we have $S' \subseteq B^{(1)}-B^{(2)}$ and
    \[\langle \mu_{A_1'}\circ\mu_{A_2'}, 1_{S'}\rangle \geq 1-\frac{\varepsilon}{16}.\]
    Thus, on applying Theorem \ref{thm:almost_periodicity} with $S'$ in place of $S$, we obtain a set $\Delta \subseteq \widehat{G}$ of size $d'$ and a parameter $\rho'$ such that 
    \[d' \ll \varepsilon^{-2}\cL(\varepsilon(\alpha_1\alpha_2)^{1/2})^2\cL(\alpha_1')\cL(\alpha_2'), \quad \rho' \gg \rho\varepsilon(\alpha_1\alpha_2)^{1/2}/(d^3d')\] 
    and $B' \vcentcolon= \mathrm{Bohr}(\Gamma\cup\Delta;\rho')$ is a regular Bohr set satisfying
    \[\lVert \mu_{A_1}*\mu_{B'}\rVert_{\infty} \geq \Bigl(1+\frac{\varepsilon}{2}\Bigr)\mu(B)^{-1},\]
    as required.
\end{proof}

By chaining together Propositions \ref{prop:few_solutions_to_lp_norm} and \ref{prop:lp_norm_to_density_increment}, we immediately obtain the following corollary.

\begin{corollary}
\label{cor:few_solutions_to_density_increment}
There is an absolute constant $c > 0$ such that the following holds. Let $A_1,A_2 \subseteq B$ be subsets of densities at least $\alpha > 0$, where $B = \mathrm{Bohr}(\Gamma;\rho)$ is a regular Bohr set of rank $d$. Let $\varepsilon \in (0,1]$ and $\lambda' \in (0,c\varepsilon\alpha/d]$ be such that $B' = B_{\lambda'}$ is a regular Bohr set. Let $C \subseteq B'$ be a subset of density $\gamma > 0$. If
\[|\langle \mu_{A_1}\circ\mu_{A_2}, \mu_C\rangle - \mu(B)^{-1}| \geq \varepsilon\mu(B)^{-1},\]
then there exist $j \in \{1,2\}$, $\Delta \subseteq \widehat{G}$ of size $\widetilde{d}$ and $\widetilde{\rho}$ such that 
\[\widetilde{d} \ll \varepsilon^{-4}\cL(\min(\varepsilon,\gamma))^2\cL(\varepsilon\alpha)^2\cL(\alpha)^2, \quad \widetilde{\rho} \geq \lambda'\rho\varepsilon\alpha/(d^5\widetilde{d})\] 
and $\widetilde{B} = \mathrm{Bohr}(\Gamma\cup\Delta;\widetilde{\rho})$ is a regular Bohr set with the property that
\[\lVert \mu_{A_j}*\mu_{\widetilde{B}}\rVert_{\infty} \geq \Bigl(1+\frac{\varepsilon}{32}\Bigr)\mu(B)^{-1}.\]
\end{corollary}
\begin{proof}
    By Lemma \ref{lm:regular_dilate}, there exist $\lambda'', \lambda''' \in [c/(2d),c/d]$ such that the Bohr sets $B'' \vcentcolon= B_{\lambda''}'$ and $B''' \vcentcolon= B_{\lambda'''}''$ are regular. In particular, $B'''$ has width $\lambda'\lambda''\lambda'''\rho \gg \lambda'(c/d)^2\rho$. We may now conclude by applying Proposition \ref{prop:few_solutions_to_lp_norm} followed by Proposition \ref{prop:lp_norm_to_density_increment}.
\end{proof}

We are now in a position to prove Proposition \ref{prop:density_increment}. In doing so, the main technical difficulties arise from working with dilates of Bohr sets on different scales, of which there are linearly many in $k$. Likewise, the need to keep track of how the density increment parameters depend on $k$ adds an extra layer of bookkeeping. This is in contrast to the situation in Roth's theorem \cite{bloom-sisask}, where the corresponding constants can be taken to be absolute.

\begin{proof}[Proof of Proposition \ref{prop:density_increment}]
    Throughout the proof, let $c,C > 0$ denote a sufficiently small and sufficiently large absolute constant, respectively, whose values may vary from line to line. Let $\delta = \delta(\frac{1}{2},\binom{k}{2}) \asymp k^{-4}$ be as in Theorem \ref{thm:graph_counting} and define the parameters
    \[\gamma \vcentcolon= 2^{-10}\delta, \quad \lambda \vcentcolon= \frac{c\alpha\gamma}{2dk}.\]
    In particular, we have $\lambda \gg k^{-5}\alpha/d$. We introduce the Bohr set $\widetilde{B} \vcentcolon= B \cap (2 \cdot B)$. Note that $\widetilde{B}$ has width $\rho$ and frequency set
    \[\widetilde{\Gamma} \vcentcolon= \Gamma \cup \{\gamma \psi_2^{-1} \mid \gamma \in \Gamma\},\]
    so its rank is at most $2d$. 
    
    In order to apply our previous results such as Corollary \ref{cor:few_solutions_to_density_increment}, we will have to work with dilates of $\widetilde{B}$ on varying scales. Specifically, each of the variables $x_1,\ldots,x_k$ will be restricted to a different dilate. Hence, using Lemma \ref{lm:regular_dilate}, we inductively choose parameters $\lambda_1, \ldots, \lambda_{k+1} > 0$ such that $\lambda_{k+1} = 1$ and for all $j \in [k]$ we have that $\frac{1}{2}\lambda\lambda_{j+1} \leq \lambda_j \leq \lambda\lambda_{j+1}$ and $B^{(j)} \vcentcolon= \widetilde{B}_{\lambda_j}$ is a regular Bohr set. In particular, for all $1 \leq i < j \leq k+1$ we have $\lambda_i \geq (\lambda/2)^{j-i}\lambda_j$. By choice of $\lambda$, we may apply \cite[Lemma 3.4]{filmus-hatami-hosseini-kelman} to the family of Bohr sets
    \[\mathcal{B} \vcentcolon= \{B^{(j)} \mid j \in [k]\} \cup \Bigl\{\frac{1}{2}\cdot B^{(j)}\ \Big|\ j \in [k]\Bigr\}\]
    to obtain an $x \in G$ such that, writing $A' \vcentcolon= A-x$, we either have $|\mu_{B'}(A')-\alpha| \leq \gamma\alpha$ for all $B'\in\mathcal{B}$ or there exists $B' \in \mathcal{B}$ such that $\mu_{B'}(A') \geq \Bigl(1+\frac{\gamma}{4k}\Bigr)\alpha$. In the latter case, we obtain the desired density increment (ii) straight away, so we may assume that the former holds.
    
    We now set the stage for an application of the graph counting lemma. Define for each $j \in [k]$ the sets $A_j \vcentcolon= A' \cap B^{(j)}$ and $\widetilde{A}_j \vcentcolon= A' \cap \Bigl(\frac{1}{2} \cdot B^{(j)}\Bigr)$ and for each $1 \leq i < j \leq k$ the set
    \[A_{i,j} \vcentcolon = \{(x,y) \in A_i \times A_j \mid x+y \in 2 \cdot \widetilde{A}_j\}.\]
    In particular, our assumption implies that $A_j$ and $2 \cdot \widetilde{A}_j$ are subsets of $B^{(j)}$ of densities belonging to the interval $[(1-\gamma)\alpha,(1+\gamma)\alpha] \subseteq [\alpha/2,2\alpha]$. Let $\alpha_{i,j}$ be the density of $A_{i,j}$ inside $A_i \times A_j$ and set $\alpha^* \vcentcolon= \min_{1\leq i <j\leq k}\alpha_{i,j}$. Then note that we have the expression
    \begin{equation}\label{eq:expression_for_density}
        \alpha_{i,j} = \langle \mu_{A_i}*\mu_{A_j}, 1_{2\cdot\widetilde{A}_j}\rangle.
    \end{equation}
    Assume to begin with that there exist $1 \leq i < j \leq k$ such that $|\alpha_{i,j}-\alpha| \geq \delta\alpha/4$. As $\gamma \leq \delta/8$, it follows from \eqref{eq:expression_for_density} that
    \begin{equation}\label{eq:discrepancy_in_3aps}
        |\langle 1_{2 \cdot \widetilde{A}_j} \circ \mu_{A_j}, \mu_{A_i} \rangle - \mu_{B^{(j)}}(2 \cdot \widetilde{A}_j)|\geq \frac{\delta\alpha}{8}.
    \end{equation}
    Since we have
    \begin{equation}\label{eq:bound_on_rel_density}
        \mu(2\cdot\widetilde{A}_j) = \mu_{B^{(j)}}(2\cdot\widetilde{A}_j)\mu(B^{(j)}) \leq 2\alpha\mu(B^{(j)}),
    \end{equation}
    dividing both sides of \eqref{eq:discrepancy_in_3aps} by $\mu(2\cdot\widetilde{A}_j)$ gives
    \begin{equation}\label{eq:starting_assumption}
        |\langle \mu_{2 \cdot \widetilde{A}_j} \circ \mu_{A_j}, \mu_{A_i} \rangle - \mu(B^{(j)})^{-1}| \geq \delta\mu(B^{(j)})^{-1}/16.
    \end{equation}
    By Corollary \ref{cor:few_solutions_to_density_increment} applied to $2\cdot\widetilde{A}_j, A_j, B^{(j)}, A_i$ in place of $A_1, A_2, B, C$ respectively and $\lambda' = \lambda_i/\lambda_j$, $\varepsilon = \delta/16$, we obtain a set $D \in \{2\cdot\widetilde{A}_j,A_j\}$, a set $\Delta \subseteq \widehat{G}$ of size $\widetilde{d}' \ll k^{16}(\log k)^4\cL(\alpha)^6$ and a parameter
    \[\widetilde{\rho}' \geq \rho\exp(-O(k(\cL(\alpha/d) + \log k)))\]
    such that $\widetilde{B}' = \mathrm{Bohr}(\widetilde{\Gamma}\cup\Delta;\widetilde{\rho}')$ is a regular Bohr set and
    \[\lVert 1_D*\mu_{\widetilde{B}'}\rVert_{\infty} \geq (1-\gamma)(1+2^{-9}\delta)\alpha \geq (1+2^{-11}\delta)\alpha.\]
    Recalling that $A_j,\widetilde{A}_j$ are subsets of a translate of $A$, we have
    \[\lVert 1_A * \mu_{\widetilde{B}'}\rVert_{\infty} \geq \lVert 1_{A_j} * \mu_{\widetilde{B}'}\rVert_{\infty},\] 
    \[\lVert 1_A * \mu_{\frac{1}{2}\cdot\widetilde{B}'}\rVert_{\infty} \geq \lVert 1_{\widetilde{A}_j} * \mu_{\frac{1}{2}\cdot\widetilde{B}'}\rVert_{\infty} = \lVert 1_{2\cdot \widetilde{A}_j} * \mu_{\widetilde{B}'}\rVert_{\infty}.\]
    This means that the alternative (ii) holds, so we are done in this case. 
    
    Hence, we may assume from now on that for all $1 \leq i < j \leq k$ we have $|\alpha_{i,j}-\alpha| \leq \delta\alpha/4$. In other words, the density of $A_{i,j}$ is roughly what one would naively expect it to be. In particular, we have $\alpha^* \geq \alpha/2$. Applying Theorem \ref{thm:graph_counting} with $\varepsilon = \frac{1}{2}$, $H = ([k], \{(i,j)\mid 1 \leq i<j \leq k\})$, $X_j = A_j$ for $j \in [k]$ and $\alpha^*$ in place of $\alpha$, we obtain that either
    \begin{equation}\label{eq:expected_count}
        \E_{(x_1,\ldots,x_k) \in A_1\times\ldots\times A_k}\prod_{1\leq i <j\leq k}1_{A_{i,j}}(x_i,x_j) \geq \frac{1}{2}\prod_{1\leq i<j \leq k}\alpha_{i,j}
    \end{equation}
    or there exist $1 \leq i < j \leq k$ such that either there exist subsets $S \subseteq A_i$, $T \subseteq A_j$ of densities at least $\exp(-O(k^6\cL(\alpha^*)^2))$ satisfying
    \begin{equation}\label{eq:density_increment_on_rect}
        \E_{(x,y)\in S\times T}1_{A_{i,j}}(x,y) \geq (1+\delta)\alpha_{i,j}
    \end{equation}
    or there exists a subset $S \subseteq A_i$ of density at least $\exp(-O((k^6\log k)\cL(\alpha^*)))$ such that for each $x \in S$ we have
    \begin{equation}\label{eq:few_neighbours}
        \E_{y\in A_j}1_{A_{i,j}}(x,y) \leq (1-\delta)\alpha_{i,j}
    \end{equation}
    or there exists a subset $S \subseteq A_i$ of density at least $\exp(-O((k^6\log k)\cL(\alpha^*)^2))$ such that for each $x \in S$ we have
    \begin{equation}\label{eq:many_neighbours}
        \E_{y\in A_j}1_{A_{i,j}}(x,y) \geq (1+\delta)\alpha_{i,j}.
    \end{equation}
    We deal with these cases in turn. First, if \eqref{eq:expected_count} holds, then by translation-invariance, we have
    \[\E_{x_1,\ldots,x_k\in G}\prod_{1\leq i\leq j\leq k}1_A\Bigl(\frac{x_i+x_j}{2}\Bigr) = \E_{x_1,\ldots,x_k\in G}\prod_{1\leq i\leq j\leq k}1_{2\cdot A'}(x_i+x_j).\]
    By passing to subsets, we may bound this from below by
    \[\E_{x_1,\ldots,x_k\in G}\prod_{j=1}^{k}1_{A_j}(x_j)\prod_{1\leq i< j\leq k}1_{2\cdot \widetilde{A}_j}(x_i+x_j) =  \prod_{j=1}^{k}\mu(A_j)\E_{(x_1,\ldots,x_k)\in A_1\times\ldots\times A_k}\prod_{1\leq i<j\leq k}1_{A_{i,j}}(x_i,x_j).\]
    Since $\mu_{B^{(j)}}(A_j) \geq (1-\gamma)\alpha$ for all $j \in [k]$ and $\alpha_{i,j} \geq (1-\delta/4)\alpha$ for all $1 \leq i<j \leq k$, we have the further lower bound
    \[\frac{1}{2}(1-\gamma)^k\Bigl(1-\frac{\delta}{4}\Bigr)^{\binom{k}{2}}\alpha^{\binom{k+1}{2}}\prod_{j=1}^{k}\mu(B^{(j)}).\]
    Furthermore, by Lemma \ref{lm:size_bohr_sets}, we have
    \[\prod_{j=1}^{k}\mu(B^{(j)}) \geq \prod_{j=1}^{k}(\lambda_j\rho/8)^{2d} \geq \prod_{j=1}^{k}((\lambda/2)^{k-j+1}\rho/8)^{2d} \geq \exp(-O((k^2\log k)d\cL(\alpha/d)))\rho^{2kd}.\]
    Since $(1-\gamma)^k, (1-\delta/4)^{\binom{k}{2}} \gg 1$, we obtain a lower bound for the density of $k$-configurations in $A$ of the same form as above, so (i) holds.
    
    Suppose now that we are in the case when \eqref{eq:many_neighbours} holds. Then for all $x \in S$ we have
    \[(1_{2\cdot \widetilde{A}_j}\circ \mu_{A_j})(x) \geq (1+\delta)\alpha_{i,j}\]
    so by averaging over $S$ and using $\alpha_{i,j} \geq (1-\delta/4)\alpha$, we obtain that
    \[\langle 1_{2\cdot \widetilde{A}_j}\circ \mu_{A_j}, \mu_S\rangle \geq \Bigl(1+\frac{\delta}{2}\Bigr)\alpha.\]
    Since $\gamma \leq \delta/4$, it follows that
    \begin{equation}\label{eq:deviation1}
        |\langle 1_{2 \cdot \widetilde{A}_j} \circ \mu_{A_j}, \mu_S \rangle - \mu_{B^{(j)}}(2 \cdot \widetilde{A}_j)|\geq \frac{\delta\alpha}{4}.
    \end{equation}
    On dividing through by $\mu(2\cdot\widetilde{A}_j)$ and using \eqref{eq:bound_on_rel_density}, we arrive at the conclusion that
    \begin{equation}\label{eq:deviation2}
        |\langle \mu_{2 \cdot \widetilde{A}_j} \circ \mu_{A_j}, \mu_S \rangle - \mu(B^{(j)})^{-1}| \geq \delta\mu(B^{(j)})^{-1}/8.
    \end{equation}
    By applying Corollary \ref{cor:few_solutions_to_density_increment} in the same way as before, except now with $\varepsilon = \delta/8$ and $S$ playing the role of $C$, we obtain a set $D \in \{2\cdot\widetilde{A}_j,A_j\}$, a set $\Delta \subseteq \widehat{G}$ of size $\widetilde{d}' \ll k^{28}(\log k)^4\cL(\alpha)^8$ and a parameter
    \[\widetilde{\rho}' \geq \rho\exp(-O(k(\cL(\alpha/d) + \log k)))\]
    such that $\widetilde{B}' = \mathrm{Bohr}(\widetilde{\Gamma}\cup\Delta;\widetilde{\rho}')$ is a regular Bohr set and
    \[\lVert 1_D*\mu_{\widetilde{B}'}\rVert_{\infty} \geq (1-\gamma)(1+2^{-8}\delta)\alpha \geq (1+2^{-10}\delta)\alpha.\]
    By the same token as before, we conclude from this that the alternative (ii) holds.

    The case when \eqref{eq:few_neighbours} holds is very similar to the previous one. Now we have
    \[(1_{2\cdot \widetilde{A}_j}\circ \mu_{A_j})(x) \leq (1-\delta)\alpha_{i,j}\]
    for all $x \in S$, so averaging and using that $\alpha_{i,j} \leq (1+\delta/4)\alpha$ yields
    \[\langle 1_{2\cdot \widetilde{A}_j}\circ \mu_{A_j}, \mu_S\rangle \leq \Bigl(1-\frac{\delta}{2}\Bigr)\alpha.\]
    Thus, we arrive at \eqref{eq:deviation1} and hence at \eqref{eq:deviation2}. Moreover, the lower bound on the density of $S$ is now larger than in the previous case, so we are again done by an application of Corollary \ref{cor:few_solutions_to_density_increment}.
    
    It remains to deal with the case when \eqref{eq:density_increment_on_rect} holds. Since $\mu_{B^{(i)}}(A_i),\mu_{B^{(j)}}(A_j) \geq \alpha/2$, we have the following lower bound on the relative densities of $S$ and $T$:
    \[\mu_{B^{(i)}}(S),\mu_{B^{(j)}}(T) \geq \exp(-O(k^6\cL(\alpha)^2)).\]
    Moreover, since $\alpha_{i,j}\geq(1-\delta/4)\alpha$, we obtain that
    \[\langle \mu_S*\mu_T, 1_{2\cdot\widetilde{A}_j}\rangle \geq \Bigl(1+\frac{\delta}{2}\Bigr)\alpha.\]
    Starting from this assumption, a suitable density increment can be deduced in a similar fashion as in \cite[pp.\ 29-31]{filmus-hatami-hosseini-kelman}. On dividing through by $\mu(2\cdot\widetilde{A}_j)$ and recalling that $\mu_{B^{(j)}}(2\cdot\widetilde{A}_j) \leq (1+\delta/8)\alpha$, one finds that
    \[\langle \mu_{2\cdot\widetilde{A}_j}\circ \mu_T, \mu_S\rangle \geq \Bigl(1+\frac{\delta}{4}\Bigr)\mu(B^{(j)})^{-1}.\]
    Hence, by Hölder's inequality, for any $p \in (1,\infty)$ we have
    \begin{align*}
        \mu_{B^{(i)}}(S)^{1-\frac{1}{p}}\lVert \mu_{2\cdot\widetilde{A}_j}\circ \mu_T\rVert_{\mathbb{L}^p(\mu_{B^{(i)}})} 
        &= \lVert \mu_{2\cdot\widetilde{A}_j}\circ \mu_T\rVert_{\mathbb{L}^p(\mu_{B^{(i)}})}\lVert 1_S\rVert_{\mathbb{L}^{p/(p-1)}(\mu_{B^{(i)}})}\\
        &\geq \langle \mu_{2\cdot\widetilde{A}_j}\circ \mu_T, 1_S\rangle_{\mathbb{L}^2(\mu_{B^{(i)}})}\\
        &= \mu_{B^{(i)}}(S)\langle \mu_{2\cdot\widetilde{A}_j}\circ \mu_T, \mu_S\rangle,
    \end{align*}
    whence we obtain that
    \[\lVert \mu_{2\cdot\widetilde{A}_j}\circ \mu_T\rVert_{\mathbb{L}^p(\mu_{B^{(i)}})} \geq \Bigl(1+\frac{\delta}{4}\Bigr)\mu_{B^{(i)}}(S)^{1/p}\mu(B^{(j)})^{-1}.\]
    For $p \geq C\delta^{-1}\cL(\mu_{B^{(i)}(S)})$, we have $(1+\delta/8)^p \geq \mu_{B^{(i)}}(S)^{-1}$ and hence
    \[\lVert \mu_{2\cdot\widetilde{A}_j}\circ \mu_T\rVert_{\mathbb{L}^p(\mu_{B^{(i)}})} \geq \Bigl(1+\frac{\delta}{16}\Bigr)\mu(B^{(j)})^{-1}.\]
    In order to apply Proposition \ref{prop:lp_norm_to_density_increment}, we must first pass to $\mathbb{L}^p$-norms with respect to a convolution of two Bohr sets. To this end, we first use Lemma \ref{lm:regular_dilate} to choose parameters $\lambda', \lambda'' \in [c/(2d),c/d]$ such that $B' \vcentcolon= B^{(i)}_{\lambda'}$ and $B'' \vcentcolon= B_{\lambda''}'$ are regular Bohr sets. By Lemma \ref{lm:domination_by_convolution}, we then have 
    \[\mu_{B^{(i)}} \leq 2\mu_{B^{(i)}_{1+2c/d}}*\nu,\]
    where we define $\nu \vcentcolon= \mu_{B'}*\mu_{B''}$. Hence, Lemma \ref{lm:averaging_argument} supplies us with a shift $x \in G$ such that
    \[\lVert \mu_{2\cdot\widetilde{A}_j}\circ \mu_T\rVert_{\mathbb{L}^p(\mu_{B'}*\mu_{B''+x})} \geq 2^{-1/p}\Bigl(1+\frac{\delta}{16}\Bigr)\mu(B^{(j)})^{-1}.\]
    If $p$ is chosen so that $p \geq 32/\delta$, we will have $(1+\delta/32)^p \geq 2$ and hence
    \[\lVert \mu_{2\cdot\widetilde{A}_j}\circ \mu_T\rVert_{\mathbb{L}^p(\mu_{B'}*\mu_{B''+x})} \geq \Bigl(1+\frac{\delta}{64}\Bigr)\mu(B^{(j)})^{-1}.\]
    Invoking Proposition \ref{prop:lp_norm_to_density_increment} with $\varepsilon = \delta/64$, $p = \lceil Ck^{10}\cL(\alpha)^2\rceil$ and $B^{(j)},B',B'',2\cdot\widetilde{A}_j,T$ in place of $B,B^{(1)},B^{(2)},A_1,A_2$ respectively, we obtain a set $\Delta \subseteq \widehat{G}$ of size $\widetilde{d}' \ll k^{46}\cL(\alpha)^{11}$ and a parameter
    \[\widetilde{\rho}' \geq \rho\exp(-O(k\cL(\alpha/d)+k^6\cL(\alpha)^2))\]
    such that $\widetilde{B}' \vcentcolon= \mathrm{Bohr}(\widetilde{\Gamma}\cup\Delta;\widetilde{\rho}')$ is a regular Bohr set and
    \[\lVert 1_{2\cdot\widetilde{A}_j}*\mu_{\widetilde{B}'}\rVert_{\infty} \geq (1-\gamma)(1+2^{-7}\delta)\alpha \geq (1+2^{-9}\delta)\alpha.\]
    Similarly as in the previous case, we infer that the alternative (ii) holds, thereby completing the proof.
\end{proof}

\section{The Erdős--Moser sum-free set problem}\label{sec:erdos_moser}

In this short section we apply our results concerning $k$-configurations to the Erdős--Moser problem. We deduce Theorem \ref{thm:erdos_moser} from Theorem \ref{thm:main_result} much in the same way as \cite[Corollary 1.4]{shao} is deduced from \cite[Theorem 1.3]{shao}. It follows from the following result, which is \cite[Proposition 2.7]{sanders} with explicit values for the implied constants.

\begin{proposition}
\label{prop:extract_sum_free}
Let $k$ be a sufficiently large positive integer and let $X \subseteq Y \subseteq \mathbb{Z}$ be such that $|X| \geq \exp(k^{68+o(1)})$ and $|Y| \leq (1+k^{-29})|X|$. Then there exists $S \subseteq X$ of size $k$ which is sum-free with respect to $Y$.
\end{proposition}

Assuming Proposition \ref{prop:extract_sum_free}, the proof of Theorem \ref{thm:erdos_moser} is the same as that of \cite[Theorem 1.2]{sanders} given \cite[Proposition 2.7]{sanders} (cf.\ the proof of \cite[Corollary 1.4]{shao} given \cite[Proposition 6.1]{shao} and also of \cite[Theorem 1.2]{sudakov-szemeredi-vu} given \cite[Theorem 1.1]{sudakov-szemeredi-vu}). Hence, we focus on establishing Proposition \ref{prop:extract_sum_free}. This will be accomplished by combining Theorem \ref{thm:main_result} with two extra ingredients, the first of which is the following.

\begin{proposition}
\label{prop:disjoint_dilate}
Let $k$, $X$ and $Y$ be as in the statement of Proposition \ref{prop:extract_sum_free} and suppose that the conclusion of Proposition \ref{prop:extract_sum_free} fails. Then there exists $X' \subseteq X$ such that $|X'| \gg k^{-29}|X|$, $|X'+X'| \ll k^{181}|X'|$ and $(2\cdot X')\cap Y = \varnothing$.
\end{proposition}

Proposition \ref{prop:disjoint_dilate} is a variant of 
\cite[Proposition 6.2]{shao} and its proof is contained in \cite[\S6]{sudakov-szemeredi-vu}; we omit the details. The second ingredient needed for the proof of Proposition \ref{prop:extract_sum_free} is Ruzsa's embedding lemma \cite[Lemma 5.1]{ruzsa-embedding}. Its use in the present context originates in the work of Shao \cite{shao}; a proof can be found in \cite[Lemma 5.26]{tao-vu}.

\begin{lemma}
\label{lm:ruzsa_embedding}
Let $A \subseteq \mathbb{Z}$ be a finite non-empty set and let $N > 4|2A-2A|$ be a positive integer. Then there exists a subset $A' \subseteq A$ such that $|A'| \geq |A|/2$ and $A'$ is Freiman $2$-isomorphic to a subset of $\mathbb{Z}/N\mathbb{Z}$, that is to say there exists a map $\varphi \colon A' \to \mathbb{Z}/N\mathbb{Z}$ such that
\begin{equation}\label{eq:freiman_isom}
    a_1+a_2 = a_1'+a_2' \iff \varphi(a_1)+\varphi(a_2) = \varphi(a_1')+\varphi(a_2')
\end{equation}
whenever $a_1,a_2,a_1',a_2' \in A'$.
\end{lemma}

We are ready to deduce Proposition \ref{prop:extract_sum_free} from Theorem \ref{thm:main_result}. In fact, it will be slightly quicker to apply Theorem \ref{thm:general_groups}, though the former would of course also suffice.

\begin{proof}[Proof of Proposition \ref{prop:extract_sum_free}]
    Assuming otherwise, Proposition \ref{prop:disjoint_dilate} provides us with a set $X' \subseteq X$ such that $|X'|\gg k^{-29}|X|$, $|X'+X'| \ll k^{181}|X'|$ and $(2\cdot X')\cap Y = \varnothing$. By the Plünnecke--Ruzsa inequality \cite[Corollary 6.29]{tao-vu}, we have $|2X'-2X'| \ll k^{724}|X'|$. Hence, by Lemma \ref{lm:ruzsa_embedding}, there exist a subset $X'' \subseteq X'$, an odd positive integer $N \ll k^{724}|X'|$ and a map $\varphi \colon X'' \to \mathbb{Z}/N\mathbb{Z}$ such that $|X''| \geq |X'|/2$ and \eqref{eq:freiman_isom} holds for all $a_1,a_2,a_1',a_2'\in X''$. In particular, $\varphi(X'')$ is a subset of $\mathbb{Z}/N\mathbb{Z}$ of density $\alpha \gg k^{-724}$ and we have
    \[N \geq |X''| \geq |X'|/2 \gg k^{-29}|X|.\]
    Provided $|X| \geq k^{29}\exp(C'k^{68}(\log k)^{16})$, where $C' > 0$ is a sufficiently large constant, Theorem \ref{thm:general_groups} implies the existence of a non-degenerate $k$-configuration in $\varphi(X'')$. By pulling back, we obtain a non-degenerate $k$-configuration in $X''$, generated by $x_1,\ldots,x_k$ say. Then for any $i,j \in [k]$ we have $x_i+x_j \in 2\cdot X'' \subseteq 2\cdot X'$ and hence $x_i+x_j \not\in Y$, so we may conclude by taking $S = \{x_1,\ldots,x_k\}$.
\end{proof}

\noindent\textbf{Acknowledgements.} This work was supported by the Croatian Science Foundation under the project number HRZZ-IP-2022-10-5116 (FANAP). The author would like to thank Rudi Mrazović for advisement and support, Zach Hunter and Rushil Raghavan for useful discussions and Tom Sanders for helpful remarks. He is also grateful to the anonymous referee for a careful reading of the manuscript and for pointing out an error in an earlier version, as well as for several helpful suggestions.

\appendix

\section{Auxiliary Kelley--Meka-type results}\label{app:aux_kelley_meka}

We collect here several results of Kelley--Meka-type that are not present in the literature in a form suitable for our applications. That being said, each of these can be obtained by making minor modifications to known results. We begin by recording an asymmetric variant of the Hölder lifting step of the Kelley--Meka argument.

\begin{proposition}
\label{prop:holder_lifting}
There is an absolute constant $c > 0$ such that the following holds. Let $\varepsilon \in (0,1]$ be a parameter and let $B \subseteq G$ be a regular Bohr set of rank $d$. Let $A_1, A_2 \subseteq B$ be subsets of densities at least $\alpha > 0$. Let $B' \subseteq B_{c\varepsilon\alpha/d}$ be a non-empty set and let $C \subseteq B'$ be a subset of density $\gamma > 0$. Then
\begin{itemize}
    \item[(i)] either $|\langle \mu_{A_1} * \mu_{A_2}, \mu_C\rangle - \mu(B)^{-1}| < \varepsilon\mu(B)^{-1}$;
    \item[(ii)] or there exists $p \ll \cL(\gamma)$ such that $\lVert (\mu_{A_1}-\mu_B)*(\mu_{A_2}-\mu_B)\rVert_{\mathbb{L}^p(\mu_{B'})} \geq \frac{1}{2}\varepsilon\mu(B)^{-1}$.
\end{itemize}
\end{proposition}
\begin{proof}
    The proof is the same as for \cite[Proposition 20]{bloom-sisask}. The result we are referring to deals only with the case $A_1 = A_2$, but the proof carries over to the general setting without any difficulty.
\end{proof}

The following lemma encodes a simple averaging argument which is used repeatedly in the paper.

\begin{lemma}
\label{lm:averaging_argument}
Let $f \colon G \to \mathbb{C}$ and $p \in [1,\infty)$. Suppose that $\mu,\nu,\eta$ are probability measures on $G$ such that $\mu \leq \gamma(\eta*\nu)$ for some $\gamma > 0$. Then there exists $x \in G$ such that
\[\lVert f\rVert_{\mathbb{L}^p(\tau_x\nu)} \geq \gamma^{-1/p}\lVert f\rVert_{\mathbb{L}^p(\mu)}.\]
\end{lemma}
\begin{proof}
    Observe that
    \[\lVert f\rVert_{\mathbb{L}^p(\mu)}^p = \langle|f|^p,\mu\rangle \leq \gamma\langle|f|^p,\eta*\nu\rangle = \gamma\langle|f|^p\circ\nu,\eta\rangle,\]
    so by averaging, we obtain an $x \in G$ such that
    \[\lVert f\rVert_{\mathbb{L}^p(\tau_x\nu)} = (|f|^p\circ\nu)(x)^{1/p} \geq \gamma^{-1/p}\lVert f\rVert_{\mathbb{L}^p(\mu)},\]
    as desired.
\end{proof}

We next give a version of the dependent random choice argument underpinning the sifting step of the Kelley--Meka proof of Roth's theorem. In essence, it can be obtained by combining the proofs of \cite[Lemma 4.9]{kelley-meka} and \cite[Lemma 8]{bloom-sisask} taken together with the remarks following its statement. We need the former for two reasons: it allows us to take different sets $A_1$, $A_2$ as input to the sifting process and produces better bounds for the densities of the outputted sets $A_1', A_2'$ (see Remark \ref{rem:alternative_sifting}). The latter is required in order to obtain a local variant of the result. For the convenience of the reader, we include a proof.

\begin{lemma}
\label{lm:asymmetric_sifting}
Let $\varepsilon,\delta \in (0,1]$ and let $p \geq \varepsilon^{-1}\cL(\delta)$ be an integer. Let $B_1, B_2 \subseteq G$ be non-empty subsets and let $A_1, A_2 \subseteq G$ be subsets of respective densities $\alpha_1,\alpha_2 > 0$. Let $\mu = \mu_{B_1}\circ\mu_{B_2}$ and
\[S = \{x \in G \mid (\mu_{A_1}\circ\mu_{A_2})(x) \geq (1-\varepsilon)\lVert \mu_{A_1}\circ\mu_{A_2}\rVert_{\mathbb{L}^p(\mu)}\}.\]
Then for $j \in [2]$ there exists a subset $A_j' \subseteq B_j$ of density at least $\frac{1}{4}(\alpha_j\lVert \mu_{A_1}\circ\mu_{A_2}\rVert_{\mathbb{L}^p(\mu)})^p$ such that
\[\langle \mu_{A_1'}\circ\mu_{A_2'},1_S\rangle \geq 1-\delta.\]
\end{lemma}
\begin{proof}
    For each $j \in [2]$ consider
    \[A_j'(t) \vcentcolon= B_j \cap \bigcap_{k=1}^{p}(A_j-t_k),\]
    where $t = (t_1,\ldots,t_p) \in G^p$. Write $\alpha_j'(t) \vcentcolon= \mu_{B_j}(A_j'(t))$ for the relative density of $A_j'(t)$ in $B_j$. Then it is straightforward to see using linearity of expectation that
    \[\E_{t\in G^p}\alpha_j'(t) = \alpha_j^p.\]
    A similar calculation reveals that for any $x \in G$ we have
    \[\E_{t\in G^p}(1_{A_1'(t)}\circ1_{A_2'(t)})(x) = (1_{A_1}\circ1_{A_2})(x)^p(1_{B_1}\circ1_{B_2})(x).\]
    Thus, for any weight function $f \colon G \to [0,\infty)$, linearity of expectation implies that
    \begin{equation}\label{eq:expected_inner_product}
        \E_{t\in G^p}\langle 1_{A_1'(t)}\circ1_{A_2'(t)}, f\rangle = \mu(B_1)\mu(B_2)\langle (1_{A_1}\circ1_{A_2})^p, f\rangle_{\mathbb{L}^2(\mu)}.
    \end{equation}
    In particular, taking $f = 1$ in the above equation, it follows that
    \[\E_{t\in G^p}\alpha_1'(t)\alpha_2'(t) = \lVert 1_{A_1}\circ1_{A_2}\rVert_{\mathbb{L}^p(\mu)}^p.\]
    Now consider the event
    \begin{equation}\label{eq:def_of_event}
        E \vcentcolon= \Bigl\{t \in G^p\ \Big|\ \alpha_j'(t) \geq \frac{1}{4}(\alpha_j\lVert \mu_{A_1}\circ\mu_{A_2}\rVert_{\mathbb{L}^p(\mu)})^p \text{ for all } j \in [2]\Bigr\}
    \end{equation}
    and observe that
    \begin{align*}
        \E_{t\in G^p}\alpha_1'(t)\alpha_2'(t)1_{E^c}(t) &\leq \frac{1}{4}(\alpha_1\lVert \mu_{A_1}\circ\mu_{A_2}\rVert_{\mathbb{L}^p(\mu)})^p\E_{t\in G^p}\alpha_2'(t) + \frac{1}{4}(\alpha_2\lVert \mu_{A_1}\circ\mu_{A_2}\rVert_{\mathbb{L}^p(\mu)})^p\E_{t\in G^p}\alpha_1'(t)\\
        &= \frac{1}{2}(\alpha_1\alpha_2\lVert \mu_{A_1}\circ\mu_{A_2}\rVert_{\mathbb{L}^p(\mu)})^p\\
        &= \frac{1}{2}\lVert 1_{A_1}\circ1_{A_2}\rVert_{\mathbb{L}^p(\mu)}^p.
    \end{align*}
    Therefore, we have
    \[\E_{t\in G^p}\alpha_1'(t)\alpha_2'(t)1_E(t) \geq \frac{1}{2}\lVert 1_{A_1}\circ1_{A_2}\rVert_{\mathbb{L}^p(\mu)}^p.\]
    The condition $p \geq \varepsilon^{-1}\cL(\delta)$ guarantees that $(1-\varepsilon)^p \leq \delta/2$, so taking $f = 1_{S^c}$ in \eqref{eq:expected_inner_product}, we get
    \[\E_{t\in G^p}\langle 1_{A_1'(t)}\circ1_{A_2'(t)},1_{S^c}\rangle < \mu(B_1)\mu(B_2)\cdot(1-\varepsilon)^p\lVert 1_{A_1}\circ1_{A_2}\rVert_{\mathbb{L}^p(\mu)}^p \leq \delta\mu(B_1)\mu(B_2)\E_{t\in G^p}\alpha_1'(t)\alpha_2'(t)1_E(t).\]
    Hence, by averaging, we obtain an element $t \in E$ such that
    \[\langle 1_{A_1'(t)}\circ1_{A_2'(t)},1_{S^c}\rangle < \delta\mu(B_1)\mu(B_2)\alpha_1'(t)\alpha_2'(t) = \delta\mu(A_1'(t))\mu(A_2'(t)).\]
    For this choice, we have the required lower bound on the densities, as well as
    \[\langle \mu_{A_1'(t)}\circ\mu_{A_2'(t)},1_S\rangle = 1-\langle \mu_{A_1'(t)}\circ\mu_{A_2'(t)},1_S^c\rangle \geq 1-\delta.\]
    This concludes the proof.
\end{proof}

\begin{remark}
\label{rem:alternative_sifting}
The choice \eqref{eq:def_of_event} of the event $E$ is based on the argument from the proof of \cite[Lemma 4.7]{kelley-meka}. We could have alternatively followed the proof of \cite[Lemma 10]{bloom-sisask} by considering
\[E \vcentcolon= \{t \in G^p \mid \alpha_1'(t)\alpha_2'(t) \geq L\}\]
with the choice $L = \frac{1}{4}(\alpha_1\alpha_2)^p\lVert \mu_{A_1}\circ\mu_{A_2}\rVert_{\mathbb{L}^p(\mu)}^{2p}$, and using the Cauchy--Schwarz inequality to bound
\[\E_{t\in G^p}\alpha_1'(t)\alpha_2'(t)1_{E^c}(t) \leq L^{1/2}\Bigl(\E_{t\in G^p}\alpha_1'(t)\Bigr)^{1/2}\Bigl(\E_{t\in G^p}\alpha_2'(t)\Bigr)^{1/2} = \frac{1}{2}\lVert 1_{A_1}\circ1_{A_2}\rVert_{\mathbb{L}^p(\mu)}^p.\]
This would ultimately lead to the same conclusion, but with the lower bound $\alpha_j' \geq L$ on the relative density of $A_j'$ in $B_j$. In the case of three-term progressions, the sets $A_1$ and $A_2$ are the same, so this would make no difference. In our case, however, this bound is weaker. Indeed, in our application, $\alpha_1$ and $\alpha_2$ are roughly proportional to $\exp(-O(\cL(\alpha)))$ and $\exp(-O(k^6\cL(\alpha)^2))$ respectively. Hence, we would get an upper bound of the form $O(p^2k^{12}\cL(\alpha)^4)$ for the product $\cL(\alpha_1')\cL(\alpha_2')$ occurring in Theorem \ref{thm:almost_periodicity}. Our argument instead delivers $O(p^2k^6\cL(\alpha)^3)$, thereby saving a factor of $k^6\cL(\alpha)$.
\end{remark}

The following almost-periodicity result is essentially \cite[Theorem 3.6]{filmus-hatami-hosseini-kelman}. In particular, it can be regarded as an asymmetric variant of \cite[Lemma 8]{bloom-sisask-improvement}. However, in contrast to these results, we use separate parameters to keep track of the densities of the input sets $A_1,A_2$, which is worthwhile in view of Remark \ref{rem:alternative_sifting}. We also spell out the dependencies of the rank and width of $B'$ on $\varepsilon$ as well as the (unsurprising) fact that its frequency set extends that of $B^{(2)}$. As in the case of Lemma \ref{lm:asymmetric_sifting}, we present a full proof.

\begin{theorem}
\label{thm:almost_periodicity}
There is an absolute constant $c > 0$ such that the following holds. Let $B \subseteq G$ be a non-empty set and let $A_1,A_2 \subseteq B$ be subsets of densities $\alpha_1,\alpha_2 > 0$ respectively. Let $B^{(1)}$ and $B^{(2)} = \mathrm{Bohr}(\Gamma;\rho)$ be regular Bohr sets of rank $d$ in $G$ such that $B^{(2)} \subseteq B^{(1)}_{c/d}$. Let $A_1' \subseteq B^{(1)}$, $A_2' \subseteq B^{(2)}$ be subsets of densities $\alpha_1',\alpha_2' > 0$ respectively. Let $\varepsilon \in (0,1]$ and let $S \subseteq B^{(1)} - B^{(2)}$ be a set such that
\begin{itemize}
    \item[(i)] $\langle \mu_{A_1'}\circ\mu_{A_2'}, 1_S\rangle \geq 1-\frac{\varepsilon}{16}$;
    \item[(ii)] $(\mu_{A_1}\circ\mu_{A_2})(x) \geq \Bigl(1+\frac{3}{4}\varepsilon\Bigr)\mu(B)^{-1}$ for all $x \in S$.
\end{itemize}
Then there exist $\Delta \subseteq \widehat{G}$ of size $d' \ll \varepsilon^{-2}\cL(\varepsilon(\alpha_1\alpha_2)^{1/2})^2\cL(\alpha_1')\cL(\alpha_2')$ and $\rho' \gg \rho\varepsilon(\alpha_1\alpha_2)^{1/2}/(d^3d')$ such that $B' = \mathrm{Bohr}(\Gamma \cup \Delta;\rho')$ is a regular Bohr set with the property that
\[\lVert \mu_{A_1}*\mu_{B'}\rVert_{\infty} \geq \Bigl(1+\frac{\varepsilon}{2}\Bigr)\mu(B)^{-1}.\]
\end{theorem}
\begin{proof}
    We start by using Lemma \ref{lm:regular_dilate} to select $\lambda \in [c/(2d),c/d]$ such that $B^{(3)} \vcentcolon= B^{(2)}_{\lambda}$ is a regular Bohr set. In particular, by Remark \ref{rem:regular_bohr_sets}, we have
    \[|{-A_2'} + B^{(3)}| \leq |{-B^{(2)}} + B^{(3)}| \leq 2|{-B^{(2)}}| = K|{-A_2'}|.\]
    for $K = 2\alpha_2'^{-1}$. Similarly, since
    \[|{-S}| \leq |B^{(1)}-B^{(2)}| \leq 2|B^{(1)}| = 2\alpha_1'^{-1}|A_1'|,\]
    it follows that $\eta \vcentcolon= \min(|A_1'|/|{-S}|,1) \geq \alpha_1'/2$. By \cite[Theorem 5.1]{schoen-sisask} applied with $-A_2',A_1',-S,B^{(3)}$ and $\varepsilon/16$ in place of $A,M,L,S$ and $\varepsilon$ respectively, we obtain a subset $X$ of a translate of  $B^{(3)}$ of density at least $\exp(-O(\varepsilon^{-2}r^2\cL(\eta)\cL(K^{-1})))$ such that
    \[\lVert \mu_X^{(r)}*\mu_{A_1'}\circ\mu_{A_2'}\circ1_S - \mu_{A_1'}\circ\mu_{A_2'}\circ1_S\rVert_{\infty} \leq \frac{\varepsilon}{16},\]
    where $r$ is a positive integer parameter, to be determined later. In particular, by (i), we have
    \[\langle \mu_X^{(r)}*\mu_{A_1'}\circ\mu_{A_2'}, 1_S\rangle \geq 1-\frac{\varepsilon}{8}\]
    and hence (ii) implies that
    \begin{equation}\label{eq:large_scalar_product}
        \langle \mu_X^{(r)}*\mu_{A_1'}\circ\mu_{A_2'}, \mu_{A_1}\circ\mu_{A_2}\rangle \geq \Bigl(1-\frac{\varepsilon}{8}\Bigr)\Bigl(1+\frac{3}{4}\varepsilon\Bigr)\mu(B)^{-1} \geq \Bigl(1+\frac{17}{32}\varepsilon\Bigr)\mu(B)^{-1}.
    \end{equation}
    Since the large spectrum is invariant under translations, the Chang--Sanders lemma \cite[Proposition 5.3]{schoen-sisask} and Lemma \ref{lm:regular_dilate} imply the existence of a set $\Delta \subseteq \widehat{G}$ of size
    \[d' \ll \varepsilon^{-2}r^2\cL(\eta)\cL(K^{-1}) \ll \varepsilon^{-2}r^2\cL(\alpha_1')\cL(\alpha_2')\]
    and a parameter $\rho' \gg \lambda\rho\nu/(d^2d')$ such that $B' \vcentcolon= \mathrm{Bohr}(\Gamma \cup \Delta;\rho')$ is a regular Bohr set with the property that $|\gamma(t)-1| \leq \nu$ for all $\gamma \in \mathrm{Spec}_{1/2}(\mu_X)$ and $t \in B'$. Here, $\nu \in (0,1]$ is a parameter to be determined later. Writing $F \vcentcolon= (\mu_{A_1'}\circ\mu_{A_2'})\circ(\mu_{A_1}\circ\mu_{A_2})$ for brevity, a standard Fourier-analytic calculation now shows that, for $t \in B'$,
    \begin{equation}\label{eq:fourier_bound}
        \lVert \tau_t(\mu_X^{(r)}*F)-\mu_X^{(r)}*F\rVert_{\infty} \leq \sum_{\gamma\in\widehat{G}}|\widehat{\mu_X}(\gamma)|^r|\widehat{F}(\gamma)||\gamma(t)-1| \leq (\nu+2^{1-r})\sum_{\gamma\in\widehat{G}}|\widehat{F}(\gamma)|.
    \end{equation}
    By the Cauchy--Schwarz inequality and Parseval's identity, we may bound
    \begin{align*}
        \sum_{\gamma\in\widehat{G}}|\widehat{F}(\gamma)| &= \sum_{\gamma\in\widehat{G}}|\widehat{\mu_{A_1'}}(\gamma)||\widehat{\mu_{A_2'}}(\gamma)||\widehat{\mu_{A_1}}(\gamma)||\widehat{\mu_{A_2}}(\gamma)|
        \leq (\alpha_1\alpha_2)^{-1/2}\mu(B)^{-1}.
    \end{align*}
    Thus, choosing $\nu = \varepsilon(\alpha_1\alpha_2)^{1/2}/64$ and $r = \lceil \log_2(2/\nu)\rceil$, it follows by combining \eqref{eq:large_scalar_product}, \eqref{eq:fourier_bound} and averaging that
    \[\langle \mu_X^{(r)}*F,\mu_{B'}\rangle \geq \Bigl(1+\frac{\varepsilon}{2}\Bigr)\mu(B)^{-1}.\]
    Finally, using the adjoint property of convolutions and Hölder's inequality, we may rewrite and bound the left-hand side as
    \[\langle \mu_X^{(r)}*\mu_{A_1'}\circ\mu_{A_2'}*\mu_{A_2},\mu_{A_1}*\mu_{B'}\rangle \leq \lVert \mu_X^{(r)}*\mu_{A_1'}\circ\mu_{A_2'}*\mu_{A_2}\rVert_1\lVert \mu_{A_1}*\mu_{B'}\rVert_{\infty},\]
    whence the desired conclusion follows.
\end{proof}

\section{Bohr sets}\label{app:bohr_sets}

We record here the definitions and properties concerning Bohr sets that are important to us; a more complete account is available in one of the standard sources in the literature, e.g.\ \cite[\S4]{tao-vu}. A large portion of what we need also appears in \cite[Appendix]{bloom-sisask}; we make appropriate references whenever possible. 

We begin by formally introducing the concept of a Bohr set.

\begin{definition}
\label{def:bohr_sets}
Given a non-empty set $\Gamma \subseteq \widehat{G}$ and a parameter $\rho \geq 0$, we define the \emph{Bohr set}
\[\text{Bohr}(\Gamma;\rho) \vcentcolon= \{x \in G \mid |\gamma(x)-1| \leq \rho \text{ for all } \gamma \in \Gamma\}.\]
$\Gamma$ is called the \emph{frequency set} and its cardinality is called the \emph{rank} of the corresponding Bohr set. We call $\rho$ the \emph{width} of the Bohr set.
\end{definition}

\begin{remark}
\label{rem:bohr_sets}
All Bohr sets are symmetric and contain $0$.
\end{remark}

It is important to note that, when speaking of a Bohr set $B$, we always implicitly fix a frequency set $\Gamma$ and a width $\rho$, and not just the ``physical'' Bohr set determined by $\Gamma$ and $\rho$ as in Definition \ref{def:bohr_sets}. Thus, on a formal level, one can think of a Bohr set purely as an ordered pair $(\Gamma,\rho)$. When omitted, the frequency set and width will always be clear from the context.

The following two definitions capture several ways of obtaining new Bohr sets from old.

\begin{definition}
\label{def:dilates_and_images}
Let $B$ be a Bohr set as in Definition \ref{def:bohr_sets}. Given a parameter $\delta \geq 0$, we define the \emph{$\delta$-dilate} of $B$ to be $B_{\delta} \vcentcolon= \text{Bohr}(\Gamma,\delta\rho)$. Given an automorphism $\psi \in \text{Aut}(G)$, we define the \emph{image} of $B$ under $\psi$ as\footnote{We suppress the function composition symbol in order to avoid confusion with difference convolution.}
\[\psi(B) \vcentcolon= \text{Bohr}(\{\gamma\psi^{-1} \mid \gamma \in \Gamma\}; \rho).\]
In particular, if $\lambda$ is an integer coprime to the order of $G$, then $\lambda \cdot B \vcentcolon= \psi_{\lambda}(B)$.
\end{definition}

\begin{remark}
\label{rem:dilates_and_images}
It is a very simple matter to check that the notion of the image of a Bohr set introduced in Definition \ref{def:dilates_and_images} is compatible with the corresponding set-theoretic concept.
\end{remark}

\begin{definition}
\label{def:intersection_bohr_sets}
If $B = \text{Bohr}(\Gamma; \rho), B' = \text{Bohr}(\Gamma';\rho')$ are Bohr sets, their intersection is defined to be
\[B \cap B' \vcentcolon= \text{Bohr}(\Gamma \cup \Gamma', \min(\rho; \rho')).\]
\end{definition}

\begin{remark}
\label{rem:intersection_bohr_sets}
The set-theoretic intersection of Bohr sets need not coincide with the physical Bohr set associated to the formal intersection. However, the former is easily seen to contain the latter. Furthermore, the operation of intersection of Bohr sets as per Definition \ref{def:intersection_bohr_sets} is commutative and associative.
\end{remark}

The next lemma provides a guarantee on the size of a Bohr set given only its rank and width; it is a variant of \cite[Lemma A.4]{bloom-sisask}.

\begin{lemma}
\label{lm:size_bohr_sets}
Let $B \subseteq G$ be a Bohr set of rank $d$ and width $\rho \in [0,2]$. Then $|B| \geq (\rho/8)^d|G|$.
\end{lemma}

In general, Bohr sets are not even approximately closed under addition. Indeed, if $B$ is a Bohr set of rank $d$, then $|B+B|/|B|$ in principle grows exponentially in $d$. However, on replacing a copy of $B$ by a narrow dilate $B'$, one can hope to recover an approximate form of closure under addition. The following definition identifies precisely the class of Bohr sets $B$ for which this is feasible.

\begin{definition}
\label{def:regular_bohr_sets}
A Bohr set $B$ of rank $d$ is called \emph{regular} if
\[(1-100\delta d)|B| \leq |B_{1-\delta}| \leq |B_{1+\delta}| \leq (1+100\delta d)|B|\]
for all $\delta \in (0,1/(100d)]$.
\end{definition}

\begin{remark}
\label{rem:regular_bohr_sets}
If $B$ is a regular Bohr set of rank $d$ and $\delta \in (0,1/(100d)]$, then $|B+B_{\delta}| \leq 2|B|$.
\end{remark}

In addition to Remark \ref{rem:regular_bohr_sets}, regularity of Bohr sets is most often exploited via a result such as \cite[Lemma A.5]{bloom-sisask}. Even though it does not appear explicitly in our paper, we do use it indirectly since it features in the proofs of \cite[Proposition 18]{bloom-sisask} and \cite[Proposition 20]{bloom-sisask}. See also Lemma \ref{lm:domination_by_convolution} below for a result in a similar spirit.

The content of the following lemma is that regular Bohr sets exist on all scales; it appears for example as \cite[Lemma A.3]{bloom-sisask}.

\begin{lemma}
\label{lm:regular_dilate}
For any Bohr set $B$, there exists $\delta \in [\frac{1}{2},1]$ such that $B_{\delta}$ is regular.
\end{lemma}

We end with a simple consequence of regularity, which is \cite[Lemma A.6]{bloom-sisask}.

\begin{lemma}
\label{lm:domination_by_convolution}
There is an absolute constant $c > 0$ such that the following holds. Let $B$ be a regular Bohr set of rank $d$ and let $k$ be a positive integer. Let $\delta \in (0,c/(kd)]$ and let $\nu$ be a probability measure supported on $kB_{\delta}$. Then
\[\mu_B \leq 2(\mu_{B_{1+k\delta}}*\nu).\]
\end{lemma}

\bibliographystyle{plain}
\bibliography{references}

\end{document}